\documentclass[11pt,a4paper]{amsart}
\usepackage{amsmath}
\usepackage{amsfonts}
\usepackage{graphicx}
\usepackage{framed}
\usepackage{amssymb}
\usepackage{mathrsfs}

\usepackage[dvipsnames]{xcolor}
\definecolor{cornellred}{rgb}{0.7, 0.11, 0.11}

\usepackage[colorlinks=true,urlcolor=RoyalBlue,citecolor=RoyalBlue,linkcolor=RoyalBlue,linktocpage,pdfpagelabels,
bookmarksnumbered,bookmarksopen]{hyperref}

\usepackage{hyperref}

\def \N {\mathbb{N}}
\def \R {\mathbb{R}}
\def \C {\mathbb{C}}

\def \d {\mathrm{d}}

\def \LL {\mathcal{L}_{\alpha}}

\def \BE {\textbf{E}}
\def \BH {\textbf{H}}
\def \BA {\textbf{A}}
\def \Bj {\textbf{j}}
\DeclareMathOperator{\Div}{div}

\usepackage{amsthm}
\usepackage{graphicx}
\usepackage{caption}
\usepackage{subcaption}
\theoremstyle{definition}
\newtheorem{definition}{Definition}[section]

\newtheorem{remark}[definition]{Remark}

\theoremstyle{plain}
\newtheorem{theorem}[definition]{Theorem}

\newtheorem{lemma}[definition]{Lemma}

\makeatletter
\@namedef{subjclassname@2020}{\textup{2020} Mathematics Subject Classification}
\makeatother

\numberwithin{equation}{section}
\begin{document}
\title[Schr\"odinger-Maxwell equations]{Schr\"odinger-Maxwell equations driven by mixed local-nonlocal operators}

\author[N.\,Cangiotti]{Nicol\`o Cangiotti}
\author[M.\,Caponi]{Maicol Caponi}
\author[A.\,Maione]{Alberto Maione}
\author[E.\,Vitillaro]{Enzo Vitillaro}

\address[N.\,Cangiotti]{Department of Mathematics\newline\indent Politecnico di Milano \newline\indent
via Bonardi 9, Campus Leonardo, 20133 Milan, Italy}
\email{nicolo.cangiotti@polimi.it}

\address[M.\,Caponi]{Dipartimento di Matematica e Applicazioni ``R. Caccioppoli''\newline\indent Università degli Studi di Napoli ``Federico II'' \newline\indent
Via Cintia, Monte S. Angelo, 80126 Naples, Italy}
\email{maicol.caponi@unina.it}

\address[A.\,Maione]{Abteilung f\"{u}r Angewandte Mathematik\newline\indent Albert-Ludwigs-Universit\"{a}t Freiburg \newline\indent Hermann-Herder-Strasse 10, 79104 Freiburg i. Br., Germany}
%\address[A.\,Maione]{Centre de Recerca Matemàtica%\newline\indent Albert-Ludwigs-Universit\"{a}t Freiburg 
%\newline\indent Edifici C, Campus Bellaterra, 08193 Bellaterra, Spain}
\email{amaione@crm.cat}

\address[E.\,Vitillaro]{Dipartimento di Matematica e Informatica DMI\newline\indent Università degli Studi di Perugia\newline\indent
Via Luigi Vanvitelli 1, 06123 Perugia, Italy}
\email{enzo.vitillaro@unipg.it}

\subjclass[2020]{35A01, 35A15, 35J50, 35J60, 35Q60, 35R11, 58E40}

\keywords{Nonlocal operators, fractional operators, variational methods, critical points theory, Schr\"odinger-Maxwell system}

\begin{abstract}
In this paper we prove existence of solutions to Schr\"odinger-Maxwell type systems involving mixed local-nonlocal operators.
Two different models are considered: classical Schr\"odinger-Maxwell equations and Schr\"odinger-Maxwell equations with a coercive potential, and the main novelty is that the nonlocal part of the operator is allowed to be nonpositive definite according to a real parameter.
We then provide a range of parameter values to ensure the existence of solitary standing waves, obtained as Mountain Pass critical points for the associated energy functionals.
\end{abstract}
\maketitle

%---------------------------------
% Introduction
%---------------------------------

\section{Introduction and main results}

Since it was introduced in the 1950s, the nonlinear Schr\"odinger equation could be considered a versatile tool in many concrete applications.
Indeed, its first appearances were strictly related to the study of superconductivity~\cite{Ginzburg1950} and superfluidity~\cite{Ginzburg1958}.
It was only later, during the 1960s, with the increasing of its physical importance, that many studies have been spreading in the physical mathematical community, however always connected with practical applications as the diffusion of optical beams~\cite{Chiao64} in a nonlinear medium.

Nowadays, many of the above quoted applications have become fundamental in many areas of Physics.
One of the most interesting problems arising from this kind of studies regards the interaction between the nonlinear Schr\"odinger equations and the electromagnetic field governed by the Maxwell equations.
Some of the models emerging from the system defined by the above type-equations have recently found applications in the electronic characterization of nanodevices~\cite{Pierantoni08,Pierantoni09}.

From a more theoretical perspective, the basic interpretation of the Schr\"odinger-Maxwell type systems is provided by the interaction between the electromagnetic field (Maxwell equations) and the particle field (Schr\"odinger equation).
In this work, we shall adopt such a point of view following the rigorous mathematical formalization firstly proposed by Benci and Fortunato~\cite{BF98}.

For the sake of completeness (with the aim of providing a clearer explanation), we briefly retrace in the following lines the mathematical argument that has led to the already classic formulation of the Schr\"odinger-Maxwell equations commonly studied in literature\footnote{Some authors refer to this system as the \emph{nonlinear Schr\"odinger-Poisson system}.
This ambiguity is motivated by the mathematical framework, in which the interacting field can be interpreted in a more general way, as for instance in terms of gravitational field.
Anyway, in this work we specifically study the electromagnetic case and so the label \emph{Maxwell} appears extremely natural.}. A similar argument can be found also in D'Avenia~\cite{dAvenia02}.

The starting point is the well-known nonlinear Schr\"odinger equation:
\begin{equation}\label{NLS}
-i\hbar\frac{\partial \psi}{\partial t}=\frac{\hbar^2}{2m}\Delta \psi +|\psi|^{p-2} \psi,
\end{equation}
with $p>2$ and where $\psi(x,t):\R^3 \times\R\to\C$ is the field describing a non-relativistic charged particle moving in the three dimensional space, $\hbar$ denotes the reduced Planck constant, and $m$ is the mass of the particle.

It is possible to interpret equation~\eqref{NLS} as the Euler-Lagrange equation with respect to the action
\[
\mathcal{S}=\iint_{\mathbb{R}^3\times\mathbb{R}}\mathcal{L}_{\text{free}} \, \d{x}\, \d{t},
\]
where
\begin{align*}
    \mathcal{L}_{\text{free}}=\frac{1}{2}\left[\hbar \left \langle i\frac{\partial \psi}{\partial t}, \psi\right \rangle - \frac{\hbar^2}{2m}|\nabla \psi|^2 \right]+\frac{1}{p}|\psi|^p,
\end{align*}
and, for any $z_j=x_j+iy_j\in \C$, $x_j, y_j \in \R$, with $j=1,2$, and $z\in\C$
\begin{align*}
    \langle z_1,z_2\rangle&=x_1x_2+y_1y_2,\\
    |z|&=\sqrt{\langle z,z\rangle}.
\end{align*}

We now introduce the electromagnetic field\footnote{The literature is divided between the two notations $\BH$ and \textbf{B} to denote the magnetic field. This is due to the two different notions of \emph{magnetic field}, which can be referred to the magnetic field strength (in ampere per meter $A/m$) or to the magnetic flux density (in tesla $T$). The relation joining these two vector quantities is determined, in the vacuum, by the \emph{vacuum permeability} $\mu_0$, indeed we have $\textbf{B}=\mu_0\BH$. In this work, following the International System of Units, we are using the symbol $\BH$ to preserve the symmetry with the electric field $\BE$ (whose unit is volt per meter $V/m$).} $(\BE,\BH)$, which can be described in terms of gauge potentials by using the first two Maxwell equations, namely
\begin{align*}
    \text{\BE} &=-\Bigg( \nabla \phi + \frac{\partial \text{\BA}}{\partial t}\Bigg),\\
    \text{\BH} &= \nabla \times \text{\BA},
\end{align*}
with
\begin{align*}
    \phi\, :& \quad \R^3 \times \R \to \R,\\
    \text{\BA}\, :& \quad \R^3\times \R \to \R^3.
\end{align*}
Let us suppose that the electromagnetic field is not assigned\footnote{The case of the assigned electromagnetic field is deeply studied.
The interested reader can see e.g.~\cite{Esteban89} for the nonlinear Schr\"odinger equation.}, so that the interaction between the fields $\psi$ and $(\BE,\BH)$ can be expressed by the \emph{rule of minimal coupling}, i.e., by substituting the ordinary derivatives with the so-called \emph{Weyl covariant derivatives}~\cite{Felsager81}:
\begin{align*}
\frac{\partial}{\partial t} \, \mapsto& \, \frac{\partial }{\partial t}+\frac{iq}{\hbar}\phi,\\
\nabla\, \mapsto& \, \nabla - \frac{iq}{\hbar}\BA.
\end{align*}
Here $q$ denotes the electrical charge.

Under these hypotheses, let us consider the following action:
\begin{align*}
    \mathcal{S}=\iint_{\mathbb{R}^3\times\mathbb{R}}\left(\mathcal{L}_{\text{int}}+ \mathcal{L}_{\text{emf}}\right) \, \d{x}\, \d{t}
\end{align*}
in which the two terms $\mathcal{L}_{\text{int}}$ and $\mathcal{L}_{\text{emf}}$ represent, respectively, the Lagrangian density of $\psi$ interacting with $(\BE, \BH)$, and the Lagrangian density of the electromagnetic field, namely
\begin{align*}
    \mathcal{L}_{\text{int}}=& \frac{1}{2}\left[\hbar \left \langle i\frac{\partial \psi}{\partial t},\psi \right \rangle - q\phi|\psi|^2-\frac{\hbar^2}{2m}\left \vert \nabla \psi - \frac{iq}{\hbar}\BA\psi\right \vert^2 \right]+\frac{1}{p}|\psi|^p,\\
    \mathcal{L}_{\text{emf}}=& \frac{1}{8\pi}\left(\left\vert\nabla \phi + \frac{\partial \BA}{\partial t}\right\vert^2-|\nabla \times \BA|^2\right).
\end{align*}

If we now take solutions of the form
\[
\psi(x,t)=u(x,t)e^{iS(x,t)/\hbar},
\]
we immediately get
\begin{align*}
    \mathcal{L}_{\text{int}}=\frac{1}{2}\left[-\frac{\hbar^2}{2m} |\nabla u|^2 - \left( \frac{\partial S}{\partial t}+q\phi+\frac{1}{2m}|\nabla S-q\BA|^2\right)u^2 \right]+\frac{1}{p}|u|^p.
\end{align*}
The Euler-Lagrange equations with respect the functional $\mathcal{S}(u,S,\phi,\BA)$ are
\begin{gather}
    -\frac{\hbar^2}{2m}\Delta u+\left( \frac{\partial S}{\partial t}+q\phi+\frac{1}{2m}|\nabla S -q \BA|^2\right)u-|u|^{p-2}u=0; \label{EL1}\\
    \frac{\partial}{\partial t}u^2+\frac{1}{m}\Div\left[(\nabla S-q\BA)u^2\right]=0;\label{EL2}\\
    qu^2=-\frac{1}{2\pi}\Div\left(\frac{\partial\BA}{\partial t}+\nabla \phi\right);\label{EL3}\\
    \frac{q}{2m}(\nabla S-q\BA)u^2=\frac{1}{4\pi}\left[ \frac{\partial }{\partial t}\left(\frac{\partial \BA }{\partial t}+\nabla \phi \right)+\nabla \times (\nabla \times \BA)\right]. \label{EL4}
\end{gather}
Let us now set
\[
\rho=\frac{qu^2}{2} \quad \text{and} \quad \Bj=\frac{q}{2m}(\nabla S-q\BA)u^2.
\]

Thus, equations~\eqref{EL2},~\eqref{EL3}, and~\eqref{EL4} become
\begin{gather}
    \frac{\partial \rho}{\partial t}+\Div\Bj=0; \label{Ce}\\[5pt]
    \Div\BE=4\pi\rho;  \label{Me1}\\
    \nabla \times \BH-\frac{\partial \BE}{\partial t}=4\pi\Bj. \label{Me2}
\end{gather}

Note that one can interpret $\rho$ as the charge density and $\Bj$ as the current density. In this view, equation~\eqref{Ce} is actually the continuity equation, and equations~\eqref{Me1} and~\eqref{Me2} are two of the Maxwell equations (precisely Gauss and Ampère equations).

Finally, since we are interested in the study of standing waves in the electrostatic case, we take
\[
u=u(x), \quad S=\omega t, \quad \phi=\phi(x), \quad \BA=0,
\]
with $\omega>0$.

This choice implies that equations~\eqref{EL2} and~\eqref{EL4} are identically satisfied, and the couple of equations~\eqref{EL1} and~\eqref{EL3} becomes the so-called Schr\"odinger-Maxwell equations:
\begin{gather*}
-\frac{\hbar^2}{2m}\Delta u+\omega u + q\phi u-|u|^{p-2}u=0,\\
-\Delta\phi=2\pi qu^2.
\end{gather*}

The mathematical formulation of the problem, exposed in the previous lines has a dual purpose.
On the one hand, we want to highlight the physical value of such a model justifying the algebraic steps providing the physical interpretation.
On the other hand, there is no doubt that these kinds of problems, arising from Physics, constitute a stimulating challenge for their intrinsic mathematical properties.

Indeed, many different questions can be addressed starting from the above equations, whether with a clear physical meaning or purely mathematical.
This scenario is pretty common, especially in the interdisciplinary domains as the Mathematical Physics.
For all these reasons, in this manuscript we try to combine two different approaches: one closely related to a physical model and one purely mathematical.
This dual perspective could offer a useful framework for those tackling this type of problems from both a physical and a mathematical point of view.

More specifically, our plan is to replace the classical Laplace operator $\Delta$ with mixed local-nonlocal operators, involving the fractional Laplacian.
Without going into technical details, we recall that the introduction of the fractional Schr\"odinger equation due to Laskin~\cite{Laskin02} (who was inspired by the Feynman path integral approach to quantum mechanics) has paved the way to a prolific field of studies based on various approaches.

Among other, we want to mention two works regarding the search of ground state solutions: the first one by Secchi~\cite{Secchi04}, who adopted a variational method based on minimization on the Nehari manifold, and the second one written by Ambrosio~\cite{Ambrosio2016}, who considered weaker assumptions (than the Ambrosetti–Rabinowitz condition) for cases where the potential is $1$-periodic or is bounded.
At the same time, our choice is also strongly motivated by an increasing interest in the study of fractional calculus \emph{per se}, as well explained by the authors in~\cite{CCMV23} (see also~\cite{GMV,MMV} for further references).
\medskip

Getting to the heart of the work, we shall deal with generalized Schr\"odinger-Maxwell (SM) type systems of the form
\begin{equation}\label{eq:problem0}
\begin{cases}
    \frac{\hbar^2}{2m}\mathcal{L}_\alpha u +\omega u+q\Phi u-|u|^{p-2}u=0&\text{in $\R^3$},\\
    -\Delta\Phi=2\pi qu^2&\text{in $\R^3$},
\end{cases}
\end{equation}
where $\hbar>0$ is the reduced Planck constant, $m>0$ is the mass of the particle, $\omega>0$, $q\in\{\pm 1\}$, and $p\in(2,2^*)$. Here $2^*$ denotes the classical Sobolev critical exponent $2^*=\dfrac{2n}{n-2}$ in dimension $n=3$, that is $2^*=6$.

The operator $\LL$ is a mixed local-nonlocal one of the following form
\begin{equation}\label{Lalfa}
    \LL = \LL^s := -\Delta +\alpha (-\Delta)^s,
\end{equation}
where $\alpha \in \R$, $\Delta$ denotes the classical Laplacian, and $(-\Delta)^s$, $s\in (0,1)$, denotes the fractional Laplacian, which we shall introduce in the sequel.

To simplify the exposition, we choose $q=1$, that is we consider the (SM) type system
\begin{equation}\label{eq:problem}
\begin{cases}
    \frac{\hbar^2}{2m}\mathcal{L}_\alpha u + u+\Phi u-|u|^{p-2}u=0&\text{in $\R^3$},\\
    -\Delta\Phi=2\pi u^2&\text{in $\R^3$}.
\end{cases}
\end{equation}
Indeed, if $(u,\Phi)$ is a solution of~\eqref{eq:problem}, then $(u,-\Phi)$ is also a solution of~\eqref{eq:problem0} with $q=-1$.

Before stating our main results, it seems appropriate to lay out a short survey of the existing literature on the topics and some of the motivations that lead us to focus on the generalization introduced in the system~\eqref{eq:problem}.
\medskip

In the recent paper~\cite{CCMV23}, the authors have generalized the  Klein-Gordon-Maxwell type systems (KGM) to the setting of mixed local-nonlocal operators, where the nonlocal one is allowed to be nonpositive definite according to a real parameter.
They provided a range of parameter values to ensure the existence of solitary waves in terms of Mountain Pass critical points for the associated energy functionals. Following the existing literature they considered two different classes of potentials: constant potentials and continuous,
bounded from below, and coercive potentials.

This paper keeps the same spirit replacing the Klein-Gordon equation with the Schr\"odinger equation. Our aim is to continue on the research line opened by the already mentioned paper of Benci and Fortunato~\cite{BF98}, which has been studied heavily alongside KGM.
Indeed, at the beginning of the 2000s Coclite and Georgiev~\cite{Coclite04}, following the original work of Benci and Fortunato, proved the existence of a sequence of radial solitary waves for these equations with a fixed $L^2$ norm and analyzed the asymptotic behavior and the smoothness of such solutions.

However, the above papers treated the linear case, while the majority of the results focused on the nonlinear case, for which important achievements have been obtained on existence, nonexistence, multiplicity, and stability.
Specifically, in 2002, D'Avenia~\cite{dAvenia02} proved the existence of non-radially symmetric wave solutions of nonlinear Schr\"odinger equation coupled with Maxwell equations.
Two years later, D'Aprile and Mugnai approached to the SM systems considering the theory of Mountain Pass critical points for the associated energy functional, in order to show the existence of radially symmetric solitary waves~\cite{Dap1} and, moreover, they obtain some non-existence results based on a suitable Pokhozhaev's identity~\cite{Dap2}.
Other impressive existence and nonexistence results were provided by Ruiz~\cite{Ruiz06} in 2006, in which he closed the gap of the previous works linked to the range of the parameter $p$. Since 2008, many authors faced the problems of SM type-systems with different kind of potentials.
Standing out among the others the works of Azzollini and Pomponio~\cite{Azzolini08}, Chen and Tang~\cite{Chen09}, and Zhao and Zhao~\cite{Zhao08}.

The purpose of the present manuscript is precisely to generalize this kind of results by considering the mixed local-nonlocal operator $\LL$, introduced in~\eqref{Lalfa}.
In a similar context, other recent papers studied the existence of solutions for analogous generalized system of equations.
In particular, the case of nonlinear fractional Schr\"odinger-Maxwell systems has been addressed by Zhang, do Ó, and Squassina~\cite{Zhang16} through a perturbation approach in the subcritical and critical case.
\smallskip

We can now proceed with the statements of the main results of this paper.

\subsection*{I. Existence results for the Schr\"odinger-Maxwell equations}

We introduce the function $\alpha_0\colon (0,1)\times (0,\infty)\to (0,\infty)$, which is defined by
\begin{align*}
	\alpha_0(s,t):=s^{-s}(1-s)^{s-1}t^{1-s}\quad\text{for $s\in(0,1)$ and $t\in(0,\infty)$},
\end{align*}
and we denote the Hilbert space
\begin{equation}\label{space_sol_Phi}
    \mathcal{D}^{1,2}(\R^3):=\overline{{\rm\bf C}_c^\infty(\R^3)}^{\|\nabla\,\cdot\,\|_2}.
\end{equation}
\begin{theorem}\label{mainthm}
Assume that
\begin{equation}\label{eq:alpha}
    \alpha>-\alpha_0\left(s,\frac{2m\omega}{\hbar^2}\right).
\end{equation}
\begin{itemize}
    \item If $p\in(4,6)$, then problem~\eqref{eq:problem} admits infinitely many radially symmetric solutions $(u_n,\Phi_n)\in H^1(\R^3)\times\mathcal{D}^{1,2}(\R^3)$.
    \item If $p=4$, then problem~\eqref{eq:problem} admits a radially symmetric solution $(u,\Phi)\in H^1(\R^3)\times\mathcal{D}^{1,2}(\R^3)$.
\end{itemize}
\end{theorem}

%--------------------------
% A purely mathematical model
%--------------------------

\subsection*{II. Existence results for the Schr\"odinger-Maxwell equations involving a coercive potential}

In the last part of the paper, motivated by the recent literature, we study the following variant of the Schr\"odinger-Maxwell system, involving coercive potentials:
\begin{equation}\label{eq:problem2}
\begin{cases}
    \frac{\hbar^2}{2m}\mathcal{L}_\alpha u +V u+\Phi u-|u|^{p-2}u=0&\text{in $\R^3$},\\
    -\Delta\Phi=2\pi u^2&\text{in $\R^3$}.
\end{cases}
\end{equation}
Here $V\colon\R^3\to\R$ satisfies:
\begin{itemize}
    \item[$(V_1)$] $V\in C(\R^3)$;
    \item[$(V_2)$] $ V_0:=\inf_{x\in\R^3}V(x)>-\infty$;
    \item[$(V_3)$] there exists $h>0$ such that	
    \begin{equation*}
        \lim_{|y|\to\infty}|\{x\in B_h(y)\,:\,V(x)\le M\}|=0\quad\text{for all $M>V_0$},
	\end{equation*}
which is trivially satisfied when $\displaystyle \lim_{|x|\to\infty}V(x)=\infty$.
\end{itemize}
Here the space of solutions for $u$ is
$$W:=\left\{u\in H^1(\R^3)\,:\, \int_{\R^3}(V-V_0)u^2\,\d x<\infty\right\}.$$
Clearly $W$ trivially reduces to $H^1(\mathbb{R}^3)$ in the main case.

The second main result of the paper is the following one.
\begin{theorem}
\label{mainthm2}
Assume the validity of conditions $(V_1)$--$(V_3)$.
\begin{itemize}
    \item If $p\in(4,6)$, then problem~\eqref{eq:problem2} admits infinitely many solutions $(u_n,\Phi_n)\in W\times\mathcal{D}^{1,2}(\R^3)$.
    \item If $p=4$, $V_0>0$, and
    \begin{equation}\label{eq:alpha2}
        \alpha>-\alpha_0\left(s,\frac{2mV_0}{\hbar^2}\right),
    \end{equation}
    then problem~\eqref{eq:problem2} admits a solution $(u,\Phi)\in W\times\mathcal{D}^{1,2}(\R^3)$.
\end{itemize}
\end{theorem}

\medskip
The paper is organized as follows. In Section~\ref{Sect2}, we outline our main assumptions, notations, and the preliminary aided further to both cases (I) and (II).
In Sections~\ref{Sect3} and~\ref{Sect4}, we shall study the cases (I) and (II), respectively, providing the proofs of Theorem~\ref{mainthm} and Theorem~\ref{mainthm2}.

%%%%%%%%%%%%%%%%%%%%%%%%%%%%%%%%%%
% Section~2
%%%%%%%%%%%%%%%%%%%%%%%%%%%%%%%%%%

%---------------------------------
% Assumptions
%---------------------------------

\section{Assumptions, notations, and preliminary results}\label{Sect2}
\subsection{Functional setting}
We recall that the Sobolev space $H^1(\R^3)$ is defined by
\begin{equation*}
H^1(\R^3)=\{u\in L^2(\R^3)\,:\,\nabla u\in L^2(\R^3;\R^3)\},
\end{equation*}
and it is a Hilbert space endowed with the norm
\begin{align*}
\|u\|_{H^1}^2:=\|u\|_2^2+\|\nabla u\|_2^2\quad\text{for $u\in H^1(\R^3)$}.
\end{align*}
We denote by $\mathscr{F}$ the Fourier transform, defined for all functions $\varphi\in\mathcal{S}(\mathbb{R}^3)$ (the Schwartz space of rapidly decreasing smooth functions) by
\begin{equation}\label{eq:Fourier}
    \mathscr{F}\varphi(\xi):=\frac{1}{(2\pi)^\frac{3}{2}}\int_{\R^3}e^{-i\langle\xi,x\rangle}\varphi(x) \,\d x\quad\text{for }\xi\in\R^3,
\end{equation}
and then extended by density to the space of tempered distributions.
By Plancharel theorem, $\mathscr{F}$ is an isometric isomorphism from $L^2(\R^3;\C)$ onto $L^2(\R^3;\C)$.

Given any $s\in(0,1)$, the fractional Sobolev space $H^s(\R^3)$ is equivalently defined by
\begin{equation*}
H^s(\R^3)=\left\{u\in L^2(\R^3)\,:\,\int_{\R^3}(1+|\xi|^{2s})|\mathscr{F}u(\xi)|^2\,\d\xi<\infty\right\},
\end{equation*}
see e.g.~\cite[Section~3]{DRV}, and it is a Hilbert space when endowed with the norm
\begin{align*}
\|u\|_{H^s}^2:=\int_{\R^3}(1+|\xi|^{2s})|\mathscr{F}u(\xi)|^2\,\d\xi\quad\text{for $u\in H^s(\R^3)$}.
\end{align*}
We notice that $H^1(\R^3)$ is continuously embedded into $H^s(\R^3)$ by Plancharel Theorem, since for all $u\in H^1(\R^3)$ we have
\begin{align*}
\int_{\R^3}|\xi|^{2s}|\mathscr{F}u(\xi)|^2\,\d\xi&\le (1-s)\int_{\R^3}|\mathscr{F}u(\xi)|^2\,\d\xi+s\int_{\R^3}|\xi|^2|\mathscr{F}u(\xi)|^2\,\d\xi\\
&=(1-s)\|u\|_2^2+s\|\nabla u\|_2^2.
\end{align*}

Let $(-\Delta)^su$ denotes the fractional Laplacian of $u$, which is defined via Fourier transform for functions $\varphi\in\mathcal{S}(\mathbb{R}^3)$ by
\[
(-\Delta)^s\varphi(x)=\mathscr{F}^{-1}(|\xi|^{2s}\mathscr{F}\varphi(\xi))(x)\quad\text{for $x\in\R^3$}.
\]
By Plancherel Theorem, we have
\begin{equation*}
H^s(\R^3)=\{u\in L^2(\R^3)\,:\,(-\Delta)^\frac{s}{2}u\in L^2(\R^3)\}
\end{equation*}
and
\begin{equation*}
\|u\|_{H^s}^2=\|u\|_2^2+\|(-\Delta)^{\frac{s}{2}}u\|_2^2.
\end{equation*}
In particular, for all $u\in H^1(\R^3)$ and $\varepsilon>0$, we have
\begin{equation}\label{young}
    \|(-\Delta)^{\frac{s}{2}}u\|_2^2=\int_{\R^3}|\xi|^{2s}|\mathscr{F}u(\xi)|^2\,\d\xi\le (1-s)\varepsilon^{-\frac{s}{1-s}}\|u\|_2^2+s\varepsilon\|\nabla u\|_2^2.
\end{equation}
Therefore, the fractional Laplacian can be interpreted as an operator
$$(-\Delta)^s\colon H^s(\R^3)\to H^{-s}(\R^3):=(H^s(\R^3))',$$
defined for all $u,v\in H^s(\R^3)$ by
\begin{equation*}
    \langle (-\Delta)^su,v\rangle_{H^{-s}(\R^3)\times H^s(\R^3)}:=\int_{\R^3}(-\Delta)^\frac{s}{2}u(-\Delta)^\frac{s}{2}v\,\d x.
\end{equation*}

\begin{remark}
We recall that the fractional Sobolev space $H^s(\R^3)$ can also be defined via the Gagliardo seminorm $[\,\cdot\,]_{s,2}$ as
\begin{equation*}
    H^s(\R^3):=\left\{u\in L^2(\R^3)\,:\,[u]_{s,2}^2:=\int_{\R^3}\int_{\R^3}\frac{|u(x)-u(x)|^2}{|x-y|^{3+2s}} \,\d x \,\d y<\infty\right\}.
\end{equation*}
Indeed, we have
\begin{equation*}
    \frac{1}{2}C(s)[u]_{s,2}^2=\int_{\R^3}|\xi|^{2s}|\mathscr{F}u(\xi)|^2\,\d\xi\quad\text{for all $u\in H^s(\R^3)$},
\end{equation*}
where the constant $C(s)$ is given by
\begin{equation}\label{eq:Cs}
    C(s):=\left(\int_{\R^3}\frac{1-\cos (x_1)}{|x|^{3+2s}} \,\d x\right)^{-1},
\end{equation}
see e.g.~\cite[Proposition~3.4 and Proposition~3.6]{DRV}. In particular, the fractional Laplacian can be defined for $\varphi\in\mathcal{S}(\R^3)$ as
\begin{equation*}
    (-\Delta)^s \varphi(x) := C(s) \,\mbox{P.V.}\int_{\R^3}\frac{\varphi(x)-\varphi(y)}{|x-y|^{3+2s}} \,\d y\quad\text{for $x\in\R^3$},
\end{equation*}
where P.V. denotes the Cauchy principal value, that is
\begin{equation*}
    \mbox{P.V.}\int_{\R^3}\frac{u(x)-u(y)}{|x-y|^{3+2s}} \,\d y:=\lim_{\varepsilon\to 0^+}\int_{\{y\in \R^3\,:\,|y-x|\ge \varepsilon\}}\frac{u(x)-u(y)}{|x-y|^{3+2s}} \,\d y,
\end{equation*}
and the constant $C(s)$ is the one defined by~\eqref{eq:Cs}.
\end{remark}

For the reader's convenience, we recall the definition of the mixed local-nonlocal operator $\mathcal{L}_\alpha$, $\alpha\in\R$, given in~\eqref{Lalfa}, that is
\begin{equation*}
\mathcal{L}_\alpha u := -\Delta u +\alpha (-\Delta)^s u.
\end{equation*}
Here $\Delta u$ denotes the classical Laplace operator, while $(-\Delta)^su$ is the fractional Laplacian.
We can then interpret $\mathcal{L}_\alpha$ as an operator
\[
\mathcal{L}_\alpha\colon H^1(\R^3)\to H^{-1}(\R^3):=(H^1(\R^3))',
\]
to which we can naturally associate a bilinear form as follows.

\begin{definition}\label{def:Balpha}
The bilinear form $\mathcal{B}_\alpha\colon H^1(\R^3)\times H^1(\R^3)\to \R$ (associated to the operator $\mathcal{L}_\alpha$) is defined, for all $u,v\in H^1(\R^3)$, by
\begin{align*}
\mathcal{B}_\alpha(u,v):=&\int_{\R^3}\langle\nabla u,\nabla v\rangle \,\d x+\alpha\int_{\R^3}(-\Delta)^\frac{s}{2}u\,(-\Delta)^\frac{s}{2}v\,\d x\\
=&\int_{\R^3}\langle\nabla u,\nabla v\rangle \,\d x+\alpha\frac{C(s)}{2}\int_{\R^3}\int_{\R^3}\frac{(u(x)-u(y))(v(x)-v(y))}{|x-y|^{3+2s}}\,\d x\,\d y.
\end{align*}
Clearly $\mathcal{B}_\alpha$ is well defined and continuous on $H^1(\R^3)\times H^1(\R^3)$.
\end{definition}

%%%%%%

\subsection{Preliminaries for the SM equations}
Regarding problem~\eqref{eq:problem}, the space of solutions for $u$ is $H^1(\R^3)$.
We recall that the embedding $H^1(\R^3)\hookrightarrow L^p(\R^3)$ is continuous for all $p\in[2,6]$, being $6=2^*$ the critical Sobolev exponent in dimension $n=3$.
In particular, for any $p\in[2,6]$, there exists a constant $C_p>0$ such that
\begin{equation}\label{eq:Sobolev}
\|u\|_p\le C_p\|u\|_{H^1}\quad\text{for all $u\in H^1(\R^3)$}.
\end{equation}

Instead, the space of solutions for $\Phi$ is the Hilbert space $\mathcal{D}^{1,2}(\R^3)$ introduced in~\eqref{space_sol_Phi}, and since in the whole space $\R^3$ the Poincar\'e inequality does not hold, we get
\begin{equation*}
    H^1(\R^3)\subsetneq\mathcal{D}^{1,2}(\R^3).
% \neq H^1_0(\R^3)=H^1(\R^3).
\end{equation*}
In any case, $\mathcal{D}^{1,2}(\R^3)$ is continuously embedded into $L^6(\R^3)$, i.e., there exists a constant $C_D>0$ such that
\begin{equation}\label{eq:D}
    \|\Phi\|_6\le C_D\|\nabla\Phi\|_2\quad\text{for all $\Phi\in\mathcal{D}^{1,2}(\R^3)$}.
\end{equation}
We can now introduce the definition of weak solutions of~\eqref{eq:problem}.
\begin{definition}\label{eq:DefWeakSol}
A pair $(u,\Phi)\in H^1(\R^3)\times\mathcal{D}^{1,2}(\R^3)$ is called a weak solution of~\eqref{eq:problem} if
\begin{equation}\label{eq:u}
    \frac{\hbar^2}{2m}\mathcal{B}_{\alpha}(u,v) + \int_{\R^3}(\omega+\Phi)uv \,\d x= \int_{\R^3}|u|^{p-2}uv\,\d x\quad\text{for all $v\in H^1(\R^3)$}
\end{equation}
and
\begin{equation}\label{eq:Phi}
    \int_{\mathbb R^3}\langle\nabla\Phi,\nabla\phi\rangle \,\d x = 2\pi\int_{\R^3}\phi u^2\,\d x\quad\text{for all $\phi\in\mathcal{D}^{1,2}(\R^3)$}.
\end{equation}
\end{definition}
To show that Definition~\ref{eq:DefWeakSol} makes sense we state and prove the following result.
\begin{lemma}\label{lem:makesense}
The system is coherent, whether $u,v\in H^1(\R^3)$ and $\Phi,\phi\in\mathcal{D}^{1,2}(\R^3)$.
\end{lemma}
\begin{proof}
Let us show that all the terms in~\eqref{eq:u} and~\eqref{eq:Phi} are well defined for $u,v\in H^1(\R^3)$ and $\Phi,\phi\in \mathcal{D}^{1,2}(\R^3)$. As observed before, the bilinear form $\mathcal{B}_\alpha$ is well defined and continuous on $H^1(\R^3)\times H^1(\R^3)$. Moreover, by H\"older inequality, we have
\begin{align*}
\left|\int_{\R^3}(\omega+\Phi)uv \,\d x\right|&\le \omega\|u\|_2\|v\|_2+\|\Phi\|_6\|u\|_\frac{12}{5}\|v\|_\frac{12}{5}<\infty,\\
\intertext{and}
\left|\int_{\R^3}|u|^{p-2}uv\,\d x\right|&\le \|u\|_p^{p-1}\|v\|_p<\infty
\end{align*}
for every $u,v\in H^1(\R^3)$ and $\Phi\in\mathcal{D}^{1,2}(\R^3)$. On the other hand
\begin{align*}
\left|\int_{\mathbb R^3}\langle\nabla\Phi,\nabla\phi\rangle \,\d x\right|&\le \|\nabla\Phi\|_2\|\nabla\phi\|_2<\infty,\\
\left|\int_{\R^3}\phi u^2\,\d x\right|&\le\|\phi\|_6\|u\|_{\frac{12}{5}}^2<\infty
\end{align*}
for every $u\in H^1(\R^3)$ and $\Phi,\phi\in\mathcal{D}^{1,2}(\R^3)$.
\end{proof}
It is easy to see that a regular solution of~\eqref{eq:problem} is actually a weak solution, according to Definition~\ref{eq:DefWeakSol}.
As usual, a weak solution of~\eqref{eq:problem} can be found by studying the critical points of the functional $F\colon H^1(\R^3)\times \mathcal{D}^{1,2}(\R^3)\to \R$, defined for any $(u,\Phi)\in H^1(\R^3)\times\mathcal{D}^{1,2}(\R^3)$ by
\[
F(u,\Phi):=\frac{\hbar^2}{4m}\mathcal{B}_\alpha(u,u)-\frac{1}{8\pi}\int_{\R^3}|\nabla\Phi|^2\,\d x +\frac{1}{2}\int_{\R^3}(\omega+\Phi)u^2\,\d x -\frac{1}{p}\int_{\R^3}|u|^p\,\d x.
\]
We remark that the functional $F$ is Fr\'{e}chet differentiable on $H^1(\R^3)\times\mathcal{D}^{1,2}(\R^3)$ and, for all $u,v\in H^1(\R^3)$ and $\Phi,\phi\in\mathcal{D}^{1,2}(\R^3)$, we have
\begin{align*}
F'_u(u,\Phi)[v] &=\frac{\hbar^2}{2m}\mathcal{B}_\alpha(u,v)+\int_{\R^3}(\omega+\Phi)uv \,\d x-\int_{\R^3}|u|^{p-2}uv \,\d x,\\
\intertext{and}
F'_\Phi (u,\Phi)[\phi] &=-\frac{1}{4\pi}\int_{\R^3}\langle\nabla\Phi,\nabla\phi\rangle \,\d x+\frac{1}{2}\int_{\R^3}\phi u^2\,\d x.
\end{align*}

Unfortunately, even though it seems to be natural to work with the functional $F$, we are unable to endow the Hilbert space $H^1(\R^3)\times\mathcal{D}^{1,2}(\R^3)$ with a norm suitable to apply the theory of critical points to $F$.
Therefore, we look for another variational characterization of problem~\eqref{eq:problem}.
\begin{lemma}\label{lem:Phiu}
For every $u\in H^1(\R^3)$ there exists a unique $\Phi(u)\in \mathcal{D}^{1,2}(\R^3)$, which is a solution of~\eqref{eq:Phi}. Moreover, we have
\begin{itemize}
    \item [(i)] $\Phi(u)\ge 0$ in $\R^3$ for all $u\in H^1(\R^3)$;
    \item [(ii)] $\Phi(tu)=t^2\Phi(u)$ for all $u\in H^1(\R^3)$ and $t\in\R$.
    \item [(iii)] If $u$ is radially symmetric, then $\Phi(u)$ is radially symmetric.
\end{itemize}
\end{lemma}
\begin{proof}
The proof of the existence and uniqueness of $\Phi(u)$ and of (i) and (ii) can be found in~\cite[Proposition~3.1]{Dap1}.

\noindent It remains to prove (iii).
We recall that a function $u\colon \R^3\to \R$ is radially symmetric, that is $u(x)=v(|x|)$ for some function $v\colon [0,\infty)\to \R$, if and only if
\[
u(g(x))=u(x)\quad\text{for all }g\in O(3)\text{ and }x\in\R^3,
\]
where $g(x):=Ox$, with $O$ orthogonal matrix.

We then fix $u\in H^1(\R^3)$ radially symmetric and $g\in O(3)$.
By the chain rule, the change of variables formula, and by~\eqref{eq:Phi}, we have
\begin{align*}
    \int_{\mathbb R^3}\langle\nabla (\Phi(u)\circ g),\nabla\phi\rangle \,\d x&=\int_{\mathbb R^3}\langle g(\nabla \Phi(u)\circ g),\nabla\phi\rangle \,\d x\\
    &=\int_{\mathbb R^3}\langle\nabla \Phi (u),\nabla (\phi\circ g^{-1})\rangle \,\d y= 2\pi\int_{\R^3}(\phi\circ g^{-1}) u^2\,\d y\\
    &=2\pi\int_{\R^3}\phi (u^2\circ g)\,\d x=2\pi\int_{\R^3}\phi u^2\,\d x
\end{align*}
for all $\phi\in \mathcal{D}^{1,2}(\R^3)$, since $\phi\circ g\in \mathcal{D}^{1,2}(\R^3)$ and $u$ is radially symmetric.

Therefore, by the uniqueness of solutions of~\eqref{eq:Phi}, we get that
\[
\Phi(u)\circ g=\Phi(u)\quad\text{for all }g\in O(3),
\]
that is, $\Phi(u)$ is radially symmetric.
\end{proof}
\begin{remark}
The unique solution $\Phi(u)$ of~\eqref{eq:Phi} can be also directly computed by convolution with the Newton potential
$$K(x):=\frac{1}{4\pi|x|}\quad\text{for all $x\in\R^3$},$$
and it has the form
\begin{align*}
\Phi(u)(x)=(K*(2\pi u^2))(x)=\frac{1}{2}\int_{\R^3}\frac{u^2(y)}{|x-y|}\,\d y\quad\text{for all $x\in\R^3$}.
\end{align*}
\end{remark}
As a consequence of Lemma~\ref{lem:Phiu}, we can deduce some useful estimates for the solutions of~\eqref{eq:Phi}.
Fix $u\in H^1(\R^3)$ and let $\Phi_u:=\Phi(u)\in\mathcal{D}^{1,2}(\R^3)$ be the unique solution of~\eqref{eq:Phi}.
Then,
\[
F_\Phi'(u,\Phi_u)[\phi]=0\quad\text{for every }\phi\in\mathcal{D}^{1,2}(\R^3)
\]
and, for $\phi=\Phi_u$, we get
\begin{equation}\label{eq:idvarphiu}
    \int_{\mathbb R^3}|\nabla\Phi_u|^2\,\d x=2\pi\int_{\R^3}\Phi_u u^2\,\d x.
\end{equation}
In particular, as a consequence of~\eqref{eq:D},~\eqref{eq:idvarphiu} and  Hölder inequality, we get
\begin{equation*}
    \|\nabla \Phi_u\|_2^2\le 2\pi\|\Phi_u\|_6\|u\|_{\frac{12}{5}}^2\le 2\pi C_D\|\nabla \Phi_u\|_2\|u\|_{\frac{12}{5}}^2,
\end{equation*}
which gives
\begin{equation}\label{eq:key2}
    \|\nabla \Phi_u\|_2\le 2\pi C_D\|u\|_{\frac{12}{5}}^2.
\end{equation}
Finally, by Lemma~\ref{lem:Phiu}--(i),~\eqref{eq:idvarphiu}, and~\eqref{eq:key2}, we have
\begin{equation}\label{eq:key}
    0\le \int_{\R^3}\Phi_u u^2\,\d x\le 2\pi C_D^2\|u\|_{\frac{12}{5}}^4.
\end{equation}

This allows us to introduce the following functional, as done in~\cite{Dap1}.
\begin{definition}\label{Definizione 2.7}
Fix any function $u\in H^1(\R^3)$, let $\Phi_u\in\mathcal{D}^{1,2}(\R^3)$ be the unique solution of~\eqref{eq:Phi}. We define the functional $J\colon H^1(\R^3)\to\R$ by
\begin{equation}\label{eq:J}
    J(u):=\dfrac{\hbar^2}{4m}\mathcal{B}_\alpha(u,u) +\dfrac{\omega}{2}\int_{\R^3}u^2\,\d x+\dfrac{1}{4}\int_{\R^3}\Phi_u u^2\,\d x-\frac{1}{p}\int_{\R^3}|u|^p \,\d x.
\end{equation}
\end{definition}
By the identity~\eqref{eq:idvarphiu}, we have
$$J(u)=F(u,\Phi_u).$$
Moreover, by standard arguments, the map $u\mapsto\Phi_u$ from $H^1(\R^3)$ into $\mathcal{D}^{1,2}(\R^3)$ is of class ${\rm\bf C}^1$.
Hence, the functional $J$ is Fr\'{e}chet differentiable on $H^1(\R^3)$ and
\begin{equation*}
J'(u)[v] =F'_u(u,\Phi_u)[v]\quad\text{for all $u,v\in H^1(\R^3)$},
\end{equation*}
since $F_\Phi'(u,\Phi_u)[\Phi_u'[v]]=0$, that is
\begin{equation}\label{gradiente}
J'(u)[v]=\dfrac{\hbar^2}{2m}\mathcal{B}_\alpha(u,v)+\omega\int_{\R^3}uv \,\d x+\int_{\R^3}\Phi_u uv \,\d x-\int_{\R^3}|u|^{p-2}uv \,\d x
\end{equation}
for any $u,v\in H^1(\R^3)$. Therefore, a pair $(u,\Phi)\in H^1(\R^3)\times\mathcal{D}^{1,2}(\R^3)$ is a weak solution of problem~\eqref{eq:problem} if and only if $\Phi=\Phi_u$ and $u$ is a critical points of $J$.

Hence, in order to find a solution of problem~\eqref{eq:problem} it is enough to find a critical point of $J$ on $H^1(\R^3)$. This is done be applying an equivariant version of the Mountain Pass Theorem (see Theorem~\ref{Rabinowitz86} below), in the case $4<p<6$, and the original Mountain Pass Theorem (see Theorem~\ref{thm:AR} below), in the case $p=4$.

\noindent In what follows, $X$ denotes an infinite dimensional Banach space and $f:X\to\mathbb{R}$.

The following notion of compactness will be required in both cases.
\begin{definition}
Let $f\in{\rm\bf C}^1(X)$.
We say that $f$ satisfies the
Palais–Smale condition, $(PS)$ condition in short, if any sequence $(u_n)_n\subset X$ such that
\begin{itemize}
    \item $(f(u_n))_n\subset\R$ is bounded,
    \item $f'(u_n)\to 0$ in $X'$ as $n\to \infty$,
\end{itemize}
has a convergent subsequence.
\end{definition}
\begin{theorem}[{\cite[Theorem 9.12]{Rabinowitz86}}]\label{Rabinowitz86}
Let $f\in {\rm\bf C}^1(X)$ be an even functional and such that $f(0) = 0$. Assume that $X$ is decomposable as direct sum of two closed subspaces $X = X_1 \oplus X_2$, with $\dim X_1 < \infty$. Suppose that:
\begin{itemize}
\item [$(i)$] there exist $\delta,\varrho>0$ such that
\begin{equation*}
\inf f(S_\varrho\cap X_2)\geq\delta,
\end{equation*}
where $S_{\varrho}:=\{u\in X\,:\,\|u\|_X=\varrho\}$;
    \item[$(ii)$] for any finite dimensional subspace $Y\subset X$ there exists $R=R(Y)>0$ such that
    \begin{equation*}
    f(u)\leq 0
    \end{equation*}
    for any $u\in Y$ with $\|u\|_X\geq R$;
    \item[$(iii)$] $f$ satisfies the $(PS)$ condition.
  \end{itemize}
  Then $f$ has an unbounded sequence of positive critical values.
\end{theorem}
\begin{theorem}[{\cite[Theorem 2.1]{AR}}]\label{thm:AR}
Let $f\in {\rm\bf C}^1(X)$ be such that $f(0) = 0$.
Assume that:
\begin{itemize}
    \item [$(i)$] there exist $\delta,\varrho>0$ such that
    \begin{equation*}
        \inf f(S_\rho)\geq\delta;
    \end{equation*}
    %where $S_{\varrho}:=\{u\in X\,:\,\|u\|_X=\varrho\}$;
    \item[$(ii)$] there exists $v\in X$ with $\|v\|_X>\delta$ such that
    \[
    f(v)\le 0;
    \]
    \item[$(iii)$] $f$ satisfies the $(PS)$ condition.
\end{itemize}
Then $f$ has a positive critical value.
\end{theorem}
\medskip

We notice that, also using Lemma~\ref{lem:Phiu}--(ii), the functional $J$ defined in~\eqref{eq:J} is even, $J\in {\rm\bf C}^1(H^1(\R^3))$ and $J(0)=0$.
In the next section we prove that a suitable restriction of $J$ satisfies the assumptions of Theorem~\ref{Rabinowitz86}, when $p\in(4,6)$, and of Theorem~\ref{thm:AR}, when $p=4$.

In order to prove the the geometric condition (ii) of Theorem~\ref{thm:AR} for $p=4$, it is convenient to introduce the following compact notation.
\begin{definition}\label{def:u-lambda}
Let $\lambda>1$ and $\beta,\gamma\in\R$ be fixed. For every $u\in L^2(\R^3)$ we define
\begin{align*}
u_{\lambda,\beta,\gamma}(x)&:=\lambda^\gamma u(\lambda^\beta x)\quad\text{for all $x\in\R^3$}.
\end{align*}
\end{definition}
By definition, $u_{\lambda,\beta,\gamma}\in L^2(\R^3)$ and, if $u\in H^1(\R^3)$, then also $u_{\lambda,\beta,\gamma}\in H^1(\R^3)$.
Moreover,
\begin{equation}\label{eq:lambda-Fourier}
    \mathscr{F} (u_{\lambda,\beta,\gamma})(\xi)=(\mathscr{F}u)_{\lambda,-\beta,\gamma-3\beta}(\xi)\quad\text{for all }u\in L^2(\R^3)\text{ and for }\xi\in\R^3.
\end{equation}
The identity~\eqref{eq:lambda-Fourier} can be proved for all $\varphi\in \mathcal{S}(\R^3)$, by using~\eqref{eq:Fourier} and the change of variable formula.
Then, it can be extended to all $u\in L^2(\R^3)$ by density.
\begin{lemma}\label{lem:stime-lambda}
Let $\lambda>1$ and $\beta,\gamma\in\R$ be fixed and, for any $u\in H^1(\R^3)$, let $u_{\lambda,\beta,\gamma}\in H^1(\R^3)$ be defined as above.
Then, for all $s\in(0,1)$, we have
\begin{align}
\|u_{\lambda,\beta,\gamma}\|_2^2&=\lambda^{2\gamma-3\beta}\|u\|_2^2,\label{eq:lam1}\\
\|\nabla (u_{\lambda,\beta,\gamma})\|_2^2&=\lambda^{2\gamma-\beta}\|\nabla u\|_2^2,\label{eq:lam2}\\
\|(-\Delta)^{\frac{s}{2}}u_{\lambda,\beta,\gamma}\|_2^2&=\lambda^{2\gamma+(2s-3)\beta}\|(-\Delta)^{\frac{s}{2}}u\|_2^2\label{eq:lam3}.
\end{align}
Moreover, if $\Phi_u$ and $\Phi_{u_{\lambda,\beta,\gamma}}$ are the unique solutions of~\eqref{eq:Phi}, respectively associated with $u$ and $u_{\lambda,\beta,\gamma}$, then
\begin{equation}
    \Phi_{u_{\lambda,\beta,\gamma}}=(\Phi_u)_{\lambda,\beta,2(\gamma-\beta)}.\label{eq:Phi-lambda}
\end{equation}
\end{lemma}
\begin{proof}
The identities~\eqref{eq:lam1} and~\eqref{eq:lam2} are a consequence of the chain rule and the change of variables formula.

\noindent To derive~\eqref{eq:lam3}, we use~\eqref{eq:lambda-Fourier} and the change of variable formula, leading to
\begin{align*}
\|(-\Delta)^{\frac{s}{2}}u_{\lambda,\beta,\gamma}\|_2^2&=\int_{\R^3}|\xi|^{2s}|\mathscr{F}(u_{\lambda,\beta,\gamma})(\xi)|^2\,\d\xi=\lambda^{2\gamma-6\beta}\int_{\R^3}|\xi|^{2s}|\mathscr{F}u(\lambda^{-\beta}\xi)|^2\,\d\xi\\
&= \lambda^{2\gamma+(2s-3)\beta}\int_{\R^3}|\hat\xi|^{2s}|\mathscr{F}u(\hat\xi)|^2\,\d\hat\xi=\lambda^{2\gamma+(2s-3)\beta}\|(-\Delta)^{\frac{s}{2}}u\|_2^2.
\end{align*}
Finally, let us prove~\eqref{eq:Phi-lambda}.
Since problem~\eqref{eq:Phi} has a unique solution $\Phi_u$ for every $u\in H^1(\R^3)$ by Lemma~\ref{lem:Phiu}, it is enough to show that $(\Phi_u)_{\lambda,\beta,2(\gamma-\beta)}$ is the solution of~\eqref{eq:Phi} associated with $u_{\lambda,\beta,\gamma}$.

Notice that for all $\phi\in \mathcal{D}^{1,2}(\R^3)$ we have that $\phi_{\lambda,-\beta,2\gamma-3\beta}\in \mathcal{D}^{1,2}(\R^3)$.
Hence, by~\eqref{eq:Phi}, we derive that
\begin{align*}
    \int_{\mathbb R^3}\langle\nabla(\Phi_u)_{\lambda,\beta,2(\gamma-\beta)}(x),\nabla\phi(x)\rangle \,\d x&=\lambda^{2\gamma-\beta}\int_{\mathbb R^3}\langle\nabla \Phi_u(\lambda^\beta x),\nabla\phi(x)\rangle \,\d x\\
%&=\lambda^{2\gamma-4\beta}\int_{\mathbb R^3}\langle\nabla(\Phi_u)(y),\nabla\phi(\lambda^{-\beta} y)\rangle \,\d y\\
    &=\int_{\mathbb R^3}\langle\nabla \Phi_u(y),\nabla(\phi_{\lambda,-\beta,2\gamma-3\beta})(y)\rangle \,\d y\\
    &= 2\pi\int_{\R^3}\phi_{\lambda,-\beta,2\gamma-3\beta}(y) u^2(y)\,\d y\\
    &=2\pi\lambda^{2\gamma-3\beta}\int_{\R^3}\phi(\lambda^{-\beta}y)u^2(y)\,\d y\\
%&=2\pi\lambda^{2\gamma}\int_{\R^3}\phi(x)u^2(\lambda^{-\beta x})\,\d x
    &=2\pi\int_{\R^3}\phi(x) u_{\lambda,\beta,\gamma}^2(x)\,\d x,
\end{align*}
which gives~\eqref{eq:Phi-lambda}.
\end{proof}

\subsection{Preliminaries for the SM equations with coercive potentials}

As done in~\cite{CCMV23} for the Klein-Gordon-Maxwell equations with potential $V$, the space of solutions for $u$ is defined as
$$W:=\left\{u\in H^1(\R^3)\,:\,\int_{\R^3}(V -V_0)u^2\,\d x<\infty\right\},$$
endowed with the norm
$$\|u\|_W^2:=\|u\|_2^2+\|\nabla u\|_2^2+\int_{\R^3}(V-V_0)u^2\,\d x,$$
while the space of solutions for $\Phi$ is the Hilbert space $\mathcal{D}^{1,2}(\R^3)$ introduced in~\eqref{space_sol_Phi}.

We recall the following result for $W$, whose proof follows from~\cite[Lemma 2.3]{CCMV23} and~\cite[Lemma 4.1]{CCMV23}.

\begin{lemma}\label{lem:com}
Assume $(V_1)$--$(V_3)$.
Then $W$ is a Hilbert space with respect to $\|\cdot\|_W$ and the space ${\rm\bf C}_c^\infty(\R^3)\subset W$ is dense in $W$. Moreover, the embedding $W\hookrightarrow L^p(\R^3)$ is continuous for all $p\in[2,6]$ and compact for all $p\in[2,6)$.
\end{lemma}
Similarly to Definition~\ref{eq:DefWeakSol}, we also introduce the notion of weak solutions of problem~\eqref{eq:problem2}.
\begin{definition}\label{eq:DefWeakSol2}
A pair $(u,\Phi)\in W\times\mathcal{D}^{1,2}(\R^3)$ is called a weak solution of~\eqref{eq:problem2} if
\begin{equation*}
    \frac{\hbar^2}{2m}\mathcal{B}_{\alpha}(u,v) + \int_{\R^3}(V+\Phi)uv \,\d x= \int_{\R^3}|u|^{p-2}uv\,\d x\quad\text{for all $v\in W$}
\end{equation*}
and
\begin{equation*}
    \int_{\mathbb R^3}\langle\nabla\Phi,\nabla\phi\rangle \,\d x = 2\pi\int_{\R^3}\phi u^2\,\d x\quad\text{for all $\phi\in\mathcal{D}^{1,2}(\R^3)$}.
\end{equation*}
\end{definition}
By the same arguments already used in  Lemma~\ref{lem:makesense},  Definition~\ref{eq:DefWeakSol2} makes sense, and every regular solution of~\eqref{eq:problem2} is actually a weak solution, according to Definition~\ref{eq:DefWeakSol2}.
As done before for the SM equations, we look for weak solutions of~\eqref{eq:problem2} as critical points of the functional $\mathcal{F}\colon W\times \mathcal{D}^{1,2}(\R^3)\to \R$ defined by
\[
\mathcal{F}(u,\Phi):=\frac{\hbar^2}{4m}\mathcal{B}_\alpha(u,u)-\frac{1}{8\pi}\int_{\R^3}|\nabla\Phi|^2\,\d x +\frac{1}{2}\int_{\R^3}(V+\Phi)u^2\,\d x -\frac{1}{p}\int_{\R^3}|u|^p\,\d x.
\]
The functional $\mathcal{F}$ is Fr\'{e}chet differentiable on $W\times\mathcal{D}^{1,2}(\R^3)$ and for all $u,v\in W$ and $\Phi,\phi\in\mathcal{D}^{1,2}(\R^3)$ we have
\begin{align*}
\mathcal{F}'_u(u,\Phi)[v] &=\frac{\hbar^2}{2m}\mathcal{B}_\alpha(u,v)+\int_{\R^3}(V+\Phi)uv \,\d x-\int_{\R^3}|u|^{p-2}uv \,\d x,\\
\mathcal{F}'_\Phi(u,\Phi)[\phi] &=-\frac{1}{4\pi}\int_{\R^3}\langle\nabla\Phi,\nabla\phi\rangle \,\d x+\frac{1}{2}\int_{\R^3}\phi u^2\,\d x.
\end{align*}
We are again unable to endow the Hilbert space $W\times\mathcal{D}^{1,2}(\R^3)$ with a norm suitable to apply the theory of critical points to $\mathcal{F}$.
We then fix $u\in W\subset H^1(\R^3)$ and consider the unique solution $\Phi_u\in\mathcal{D}^{1,2}(\R^3)$ of~\eqref{eq:Phi}, given by Lemma~\ref{lem:Phiu}.
Then,
\[
\mathcal{F}_\Phi'(u,\Phi_u)[\phi]=0
\]
for every $\phi\in\mathcal{D}^{1,2}(\R^3)$, and for $\phi=\Phi_u$ we get~\eqref{eq:idvarphiu}.
Therefore, we can introduce the following functional.

\begin{definition}
Fix any function $u\in W$, let $\Phi_u\in\mathcal{D}^{1,2}(\R^3)$ be the unique solution of~\eqref{eq:Phi}. We define the functional $\mathcal{J}\colon W\to\R$ by
\begin{equation}\label{eq:J2}
\mathcal{J}(u):=\dfrac{\hbar^2}{4m}\mathcal{B}_\alpha(u,u) +\dfrac{1}{2}\int_{\R^3}Vu^2\,\d x+\dfrac{1}{4}\int_{\R^3}\Phi_u u^2\,\d x-\frac{1}{p}\int_{\R^3}|u|^p \,\d x.
\end{equation}
\end{definition}
By the identity~\eqref{eq:idvarphiu}, we have
$$\mathcal{J}(u)=\mathcal{F}(u,\Phi_u),$$
and, by standard arguments, the map $u\mapsto\Phi_u$ from $W$ into $\mathcal{D}^{1,2}(\R^3)$ is of class ${\rm\bf C}^1$. Hence, the functional $\mathcal{J}$ is Fr\'{e}chet differentiable on $W$ and
\begin{equation*}
\mathcal{J}'(u)[v] =\mathcal{F}_u'(u,\Phi_u)[v]\quad\text{for all $u,v\in W$},
\end{equation*}
since $\mathcal{F}_\Phi'(u,\Phi_u)[\Phi_u'[v]]=0$, that is, for any $u,v\in W$
\begin{align*}
\mathcal{J}'(u)[v]&=\dfrac{\hbar^2}{2m}\mathcal{B}_\alpha(u,v)+\int_{\R^3}Vuv \,\d x+\int_{\R^3}\Phi_u uv \,\d x-\int_{\R^3}|u|^{p-2}uv \,\d x.
\end{align*}
Therefore, a pair $(u,\Phi)\in W\times\mathcal{D}^{1,2}(\R^3)$ is a weak solution of problem~\eqref{eq:problem2} if and only if $\Phi=\Phi_u$ and $u$ is a critical points of $\mathcal{J}$. Thus, to find a solution of problem~\eqref{eq:problem2}, we search critical points of $\mathcal{J}$ on $W$, and this is done be applying Theorem~\ref{Rabinowitz86}, in the case $4<p<6$, and Theorem~\ref{thm:AR}, in the case $p=4$.
Indeed, by Lemma~\ref{lem:Phiu}--(ii), the functional $\mathcal{J}\colon W\to\R$ defined in~\eqref{eq:J2} is even, and it satisfies $\mathcal{J}\in {\rm\bf C}^1(W)$ and $\mathcal{J}(0)=0$.

%-------------------------------------
% SM equation
%-------------------------------------

\section{The SM equations}
\label{Sect3}
To find critical points of the functional $J$ defined in~\eqref{eq:J}, we shall restrict it to the subspace of radial functions
\[H_r^1(\R^3):=\{u\in H^1(\R^3)\,:\,u(x)=v(|x|)\text{ for any }x\in\R^3\}.\]
This (standard) procedure is allowed by the following result.
\begin{lemma}\label{Lemmaradial}
Under the assumptions of Theorem~\ref{mainthm}, $u\in H^1_r(\mathbb{R}^3)$ is a critical point of $J|_{H^1_r(\mathbb{R}^3)}$ if and only if $u$ is a critical point of $J$.
\end{lemma}
\begin{proof}
    The arguments of the proof of~\cite[Lemma~3.1]{CCMV23} apply as well for the functional $J$ defined in~\eqref{eq:J}.
\end{proof}
\begin{lemma}\label{condgeome1}
Assume the validity of the assumptions of Theorem~\ref{mainthm}. Then
\begin{itemize}
    \item when $p\in(4,6)$, the functional $J$ satisfies $(i)$ and $(ii)$ of Theorem~\ref{Rabinowitz86} in $X=H^1_r(\R^3)$, with $X_1=\{0\}$ and $X_2=X$;
    \item when $p=4$, the functional $J$ satisfies $(i)$ and $(ii)$ of Theorem~\ref{thm:AR} in $X=H^1_r(\R^3)$.
\end{itemize}
\end{lemma}
\begin{proof}
We divide the proof in two steps.

\noindent{\bf Step 1.} The first geometric condition in Theorems~\ref{Rabinowitz86} and~\ref{thm:AR}.

We claim that for every $p\in[4,6)$ there exist $\delta,\varrho>0$ such that
\begin{equation*}
\inf J(S_\varrho) \ge \delta,
\end{equation*}
where $S_{\varrho}:=\{u\in H^1_r(\R^3) \,:\, \|u\|_{H^1}=\varrho\}$.
Indeed, by~\eqref{young}, for any $u\in H^1_r(\R^3)$ and $\varepsilon>0$, we have
\begin{align*}
    &\dfrac{\hbar^2}{4m}\mathcal{B}_\alpha(u,u)+\frac{\omega
    }{2}\|u\|^2_2\\
    &\ge \dfrac{\hbar^2}{4m}\|\nabla u\|_2^2-\dfrac{\alpha^-\hbar^2}{4m}\left(s\varepsilon \|\nabla u\|_2^2+(1-s)\varepsilon^{-\frac{s}{1-s}}\|u\|^2_2\right)+\frac{\omega
    }{2}\|u\|^2_2\\
    &= \dfrac{\hbar^2}{4m}\left(1-\alpha^{-}s\varepsilon\right)\|\nabla u\|_2^2+\frac{1}{2}\left(\omega-\dfrac{\alpha^-\hbar^2}{2m}(1-s)\varepsilon^{-\frac{s}{1-s}}\right)\|u\|^2_2,
\end{align*}
where $\alpha^-:=\max\{-\alpha,0\}$ denotes the negative part of $\alpha$.
Let us consider the following system
\begin{equation}\label{eq:system}
\begin{cases}
1-\alpha^-s\varepsilon>0,\\
\omega-\dfrac{\alpha^-\hbar^2}{2m}(1-s)\varepsilon^{-\frac{s}{1-s}}>0.
\end{cases}
\end{equation}
The first inequality of~\eqref{eq:system} holds whenever
\begin{equation*}
\alpha^-s\varepsilon<1,
\end{equation*}
while the second inequality of~\eqref{eq:system} leads us to
\begin{align*}
\alpha^-(1-s)\varepsilon^{-\frac{s}{1-s}}&<\frac{2m\omega}{\hbar^2}.
\end{align*}
Therefore, the system~\eqref{eq:system} is satisfied when either $\alpha^-=0$ or $\alpha^->0$ and
$$\frac{(1-s)^{\frac{1-s}{s}}\left(\hbar^2\right)^{\frac{1-s}{s}}\left(\alpha^-\right)^{\frac{1-s}{s}}}{(2m\omega)^{\frac{1-s}{s}}}<\varepsilon<\frac{1}{\alpha^-s}.$$
Since $\alpha$ satisfies~\eqref{eq:alpha}, which implies
\begin{equation*}
\alpha^{-}<\left(\frac{2m\omega}{\hbar^2}\right)^{1-s}(1-s)^{s-1}s^{-s},
\end{equation*}
there exists $\varepsilon_0\in(0,\infty)$ such that
\begin{equation}\label{eq:c1c2}
c_1:=\frac{\hbar^2}{2m}(1-\alpha^{-}s\varepsilon_0)>0\quad\text{and}\quad c_2:=\omega-\dfrac{\alpha^-\hbar^2}{2m}(1-s)\varepsilon_0^{-\frac{s}{1-s}}>0.
\end{equation}
Hence we get
\begin{align}\label{eq:below}
    \dfrac{\hbar^2}{4m}\mathcal{B}_\alpha(u,u)+\frac{\omega
    }{2}\|u\|^2_2\ge& \frac{1}{2}\min\{c_1,c_2\}\|u\|^2_{H^1}.
\end{align}
Therefore, by also using Lemma~\ref{lem:Phiu}--(i) and~\eqref{eq:Sobolev}, for any $u\in S_{\varrho}$ we deduce
\begin{align*}
J(u)=&\dfrac{\hbar^2}{4m}\mathcal{B}_\alpha(u,u)+\frac{\omega
    }{2}\|u\|^2_2+\dfrac{1}{4}\int_{\R^3}\Phi_u u^2\,\d x-\frac{1}{p}\|u\|^p_p\\
    \ge& \frac{1}{2}\min\{c_1,c_2\}\|u\|^2_{H^1}-\frac{C_p^p}{p}\|u\|^p_{H^1}\\
    =& \frac{1}{2}\min \{c_1,c_2\} \cdot \varrho^2-\frac{C_p^p}{p} \cdot \varrho^p\\
    =&\varrho^2 \left( \frac{1}{2}\min \{c_1,c_2\}-\frac{C_p^p}{p}\cdot \varrho^{p-2}\right)>0,
\end{align*}
where the last inequality holds for
\begin{equation*}
  \varrho<\left(\frac{p\min \{c_1,c_2\}}{2C_p^p}\right)^\frac{1}{p-2}.
\end{equation*}
Thus condition $(i)$ is satisfied for both Theorems~\ref{Rabinowitz86} and~\ref{thm:AR}.
\medskip

\noindent{\bf Step 2.} The second geometric condition in Theorems~\ref{Rabinowitz86} and~\ref{thm:AR}. We consider two cases.

\noindent{\bf Case 1.}
For $p\in(4,6)$, we fix a finite dimensional space $Y\subset H_r^1(\R^3)$ and $u\in Y$.

By the continuity of $\mathcal{B}_\alpha$,~\eqref{young},~\eqref{eq:key} and~\eqref{eq:J}, there exists a positive constant $K>0$ such that
\begin{equation*}
J(u)\le K\|u\|_{H^1}^2+\frac{\pi}{2}C_D^2\|u\|_{\frac{12}{5}}^4-\frac{1}{p}\|u\|_p^p\to-\infty
\end{equation*}
as $\|u\|_{H^1}\to \infty$, since on $Y$ all norms are equivalent. Therefore, $(ii)$ of Theorem~\ref{Rabinowitz86} is satisfied for $p\in(4,6)$.

\noindent{\bf Case 2.}
Let $p=4$ and, for every $\lambda>1$, $\beta,\gamma\in\R$ and $u\in H^1_r(\R^3)$, let $u_{\lambda,\beta,\gamma}$ be as in Definition~\ref{def:u-lambda}.
We first remark that, by a simple change of variable, we get
\[
\int_{\R^3}\Phi_{u_{\lambda,\beta,\gamma}}u_{\lambda,\beta,\gamma}^2\,\d x=\lambda^{4\gamma-5\beta}\int_{\R^3}\Phi_u u^2\,\d x \quad\text{and}\quad \|u_{\lambda,\beta,\gamma}\|_4^4=\lambda^{4\gamma-3\beta}\|u\|_4^4.
\]
Then,  also using Lemma~\ref{lem:stime-lambda}, for any $u\in H^1_r(\R^3)\setminus\{0\}$ we have
\begin{align*}
J(u_{\lambda,\beta,\gamma})\le &\frac{\hbar^2}{4m}\lambda^{2\gamma-\beta}\|\nabla u\|_2^2+\frac{\alpha^+\hbar^2}{4m}\lambda^{2\gamma+2s\beta-3\beta}\|(-\Delta)^{\frac{s}{2}}u\|_2^2\\
&+\frac{\omega}{2}\lambda^{2\gamma-3\beta}\|u\|_2^2+\frac{1}{4}\lambda^{4\gamma-5\beta}\int_{\R^3}\Phi_u u^2\,\d x-\frac{1}{4}\lambda^{4\gamma-3\beta}\|u\|_4^4.
\end{align*}
We want to prove that $J(u_{\lambda,\beta,\gamma})\le 0$ for some suitable choice of $\lambda>1$ and $\beta,\gamma\in\R$.

For example, if we assume the validity of the following system
\begin{equation}\label{system}
\begin{cases}
4\gamma-3\beta>0,\\
4\gamma-3\beta>4\gamma-5\beta,\\
4\gamma-3\beta>2\gamma-3\beta,\\
4\gamma-3\beta>2\gamma+2s\beta-3\beta,\\
4\gamma-3\beta>2\gamma-\beta,
\end{cases}
\end{equation}
then $J(u_{\lambda,\beta,\gamma})\to -\infty$ as $\lambda\to \infty$. We then look for couples $(\beta,\gamma)$ satisfying~\eqref{system}.
From the third inequality, we must take $\gamma>0$. The second and the fifth inequalities imply
$$0<\frac{\beta}{\gamma}<1,$$
which is satisfied for $\gamma=2\beta$. This choice satisfies also the first and the forth inequalities, being $s\in(0,1)$.

Notice that, by Lemma~\ref{lem:stime-lambda}, we have
$$\|u_{\lambda,\beta,2\beta}\|_{H^1}^2=\lambda^\beta\|u\|_2^2+\lambda^{3\beta}\|\nabla u\|_2^2\to \infty\quad\text{as $\lambda\to\infty$},$$
therefore condition $(ii)$ of Theorem~\ref{thm:AR} is satisfied for $p=4$.
\end{proof}
\begin{lemma}\label{PS1}
Under the assumptions of Theorem~\ref{mainthm}, the functional $J|_{H^1_r(\R^3)}$ satisfies $(iii)$ in Theorems~\ref{Rabinowitz86} and~\ref{thm:AR}.
\end{lemma}
\begin{proof}
Let us fix a $(PS)$ sequence $(u_n)_n\subset H^1_r(\R^3)$.
Then, by Definition~\ref{Definizione 2.7},~\eqref{gradiente} and~\eqref{eq:below}, for all $p\in[4,6)$ we have
\begin{equation}\label{case1lower}
\begin{aligned}
    &pJ(u_n)-J'(u_n)[u_n]\\
    &=\left(\frac{p}{2}-1\right)\left(\dfrac{\hbar^2}{2m}\mathcal{B}_\alpha(u_n,u_n)+\omega\|u_n\|^2_2\right)+\left(\frac{p}{4}-1\right)\int_{\R^3}\Phi_{u_n}u_n^2\,\d x\\
    &\ge\left(\frac{p}{2}-1\right)\left(\dfrac{\hbar^2}{2m}\mathcal{B}_\alpha(u_n,u_n)+\omega\|u_n\|^2_2\right)\\
    &\ge \left(\frac{p}{2}-1\right)\min\{c_1,c_2\}\|u_n\|_{H^1}^2,
\end{aligned}
\end{equation}
where $c_1>0$ and $c_2>0$ are the two constants defined in~\eqref{eq:c1c2}.

Since $(J(u_n))_n$ is bounded in $\R$ and $(J'|_{H_r^1(\R^3)}(u_n))_n$ is bounded in $(H^1_r(\R^3))'$, being $(u_n)_n$ a $(PS)$ sequence, there exist two positive constants $K_1,K_2$ such that
\begin{equation}\label{PSseq}
    J(u_n)\le\ K_1\quad\text{and}\quad|J'(u_n)[u_n]|\le K_2\|u_n\|_{H^1}\quad\text{for any }n\in\mathbb{N}.
\end{equation}
Hence, by~\eqref{case1lower} and~\eqref{PSseq}, we get
\begin{equation*}
    pK_1+K_2\|u_n\|_{H^1}\ge \left(\frac{p}{2}-1\right)\min\{c_1,c_2\}\|u_n\|_{H^1}^2\quad\text{for any }n\in\mathbb{N},
\end{equation*}
which implies that $(u_n)_n$ is bounded in $H^1_r(\R^3)$.

Therefore, there exists a subsequence, not relabeled, and $u\in H_r^1(\R^3)$ such that $(u_n)_n$ converges to $u$ weakly in $H_r^1(\R^3)$, strongly in $L^p(\R^3)$ for any $p\in(2,6)$, and a.e. in $\R^3$.
To conclude, we show that the convergence in $H_r^1(\R^3)$ turns out to be strong.

By~\eqref{eq:key2}, the sequence $(\Phi_{u_n})_n\subset\mathcal{D}^{1,2}(\R^3)$ is bounded in $\mathcal{D}^{1,2}(\R^3)$, being $(u_n)_n$ bounded in $H^1_r(\R^3)$, and $2<\frac{12}{5}<6=2^*$
(see Lemma~\ref{lem:makesense} for further details).

Moreover, by~\eqref{gradiente} and~\eqref{eq:below}, we have
\begin{equation}\label{stimePS}
\begin{split}
        \min\{c_1,c_2\}\|u_n-u\|_{H^1}^2&\le\dfrac{\hbar^2}{2m}\mathcal{B}_{\alpha}(u_n-u,u_n-u)+\omega\|u_n-u\|^2_2\\
    &=J'(u_n)[u_n-u]-J'(u)[u_n-u]\\
    &\quad-\int_{\R^3}(\Phi_{u_n}u_n-\Phi_uu)(u_n-u)\,\d x\\
    &\quad+\int_{\R^3}(|u_n|^{p-2}u_n-|u|^{p-2}u)(u_n-u)\,\d x.
\end{split}
\end{equation}

Since $(J'(u_n))_n$ strongly converges to $0$ in $(H^1_r(\R^3))'$ and $(u_n)_n$ weakly converges to $u$  in $H^1_r(\R^3)$ as $n\to\infty$, the first two terms on the right hand side of~\eqref{stimePS} converge to 0 as $n\to\infty$.
Moreover, as in Lemma~\ref{lem:makesense},
\begin{align*}
    \left|\int_{\R^3}(\Phi_{u_n}u_n-\Phi_u u)(u_n-u)\,\d x\right|\le\left(\|\Phi_{u_n}\|_6\|u_n\|_\frac{12}{5}+\|\Phi_u\|_6\|u\|_\frac{12}{5}\right)\|u_n-u\|_\frac{12}{5}&,\\
    \left|\int_{\R^3}(|u_n|^{p-2}u_n-|u|^{p-2}u)(u_n-u)\,\d x\right|\le(\|u_n\|_p^{p-1}+\|u_n\|_p^{p-1})\|u_n-u\|_p&,
\end{align*}
that is,
\begin{align*}
    \left|\int_{\R^3}(\Phi_{u_n}u_n-\Phi_u u)(u_n-u)\,\d x\right|&\to 0,\\
    \left|\int_{\R^3}(|u_n|^{p-2}u_n-|u|^{p-2}u)(u_n-u)\,\d x\right|&\to 0,
\end{align*}
as $n\to \infty$, and the thesis follows.
\end{proof}
We can now conclude this section by collecting all the results given in the previous lines to prove Theorem~\ref{mainthm}.
\begin{proof}[Proof of Theorem~\ref{mainthm}]
The proof follows by Lemmas~\ref{Lemmaradial}--\ref{PS1} and by Theorems~\ref{Rabinowitz86} and~\ref{thm:AR}.
\end{proof}

%-------------------------------------
% SM equation with coercive potentials
%-------------------------------------

\section{The SM equations with coercive potentials}\label{Sect4}

In order to prove that, for $p\in(4,6)$, the functional $\mathcal{J}$ defined in~\eqref{eq:J2} satisfies the geometric conditions $(i)$ and $(ii)$ of Theorem~\eqref{Rabinowitz86}, we use arguments similar to those in~\cite[Section~4]{CCMV23}.
We introduce a new operator $\mathcal{L}_{\alpha,V}\colon W\to W'$, defined by
$$\mathcal{L}_{\alpha,V}u=\frac{\hbar^2}{2m}\mathcal{L}_\alpha u+Vu\quad\text{for $u\in W$}.$$
As done in Definition~\ref{def:Balpha}, we can naturally associate to $\mathcal{L}_{\alpha,V}$ a bilinear form $\mathcal{B}_{\alpha,V}\colon W\times W\to \R$, defined by
\begin{align*}
    \mathcal{B}_{\alpha,V}(u,v):=\frac{\hbar^2}{2m}\mathcal{B}_\alpha(u,v)+\int_{\R^3}Vuv\,\d x\quad\text{for all $u,v\in W$}.
\end{align*}
We remark that $\mathcal{B}_{\alpha,V}$ is continuous on $W\times W$ and, by~\eqref{young}, for all $\alpha\in\R$ there exists a constant $\mu =\mu (s,\alpha,V_0)\ge 0$ such that
\begin{equation}\label{eq:below2}
    \mathcal{B}_{\alpha,V}(u,u)+\mu \|u\|_2^2\ge \frac{1}{2}\|u\|_W^2\quad\text{for all }u\in W.
\end{equation}
Since, by Lemma~\ref{lem:com}, the embedding $W\hookrightarrow L^2(\R^3)$ is continuous, dense, and compact, we can apply the spectral decomposition result given in~\cite[Proposition~A.4]{CCMV23}.
Hence, there exists an increasing sequence $(\lambda_k)_k$ of eigenvalues of $\mathcal{L}_{\alpha,V}$ satisfying
\begin{equation*}
    -\mu < \lambda_1\le \lambda_2\le\cdots\le \lambda_k\to \infty\quad\text{as $k\to\infty$}.
\end{equation*}
Moreover, for all $k\in \mathbb N$, the eigenvalue $\lambda_k$ has finite multiplicity and there exists a sequence of eigenvectors $(e_k)_k\subset W$, corresponding to $(\lambda_k)_k$, which is an orthonormal basis of $L^2(\R^3)$.
We define the spaces
\begin{align*}
    H_1&:=\{0\},\\
    H_k&:=\textrm{span}\{e_1,\dots,e_{k-1}\}\subset W\text{ for all $k\ge 2$},
\end{align*}
and
\begin{align*}
    \mathbb P_1&:=W,\\
    \mathbb P_k&:=\left\{u\in W:\int_{\R^3}ue_j=0\text{ for all $j=1,\dots,k-1$}\right\}\text{ for all }k\ge 2.
\end{align*}
Then $W$ is decomposable as the direct sum of these two closed subspace, that is as  $W=H_k\oplus\mathbb P_k$ for all $k\in\N$, with $\dim H_k=k-1<\infty$, and
\begin{equation}\label{eq:lambdak}
    \lambda_k:=\min_{u\in \mathbb P_k\setminus\{0\}}\frac{\mathcal{B}_{\alpha,V}(u,u)}{\|u\|_2^2}.
\end{equation}

Let $k_0\in\N$ be such that
\begin{equation}\label{k0}
    \lambda_{k_0}>0.
\end{equation}
In view of~\eqref{eq:below2} and~\eqref{eq:lambdak}, there exists a constant $c_0=c_0(s,\alpha,V_0)>0$ satisfying
\begin{align}\label{eq:ck0}
\mathcal{B}_{\alpha,V}(u,u)\ge c_0\|u\|_W^2\quad\text{for all $u\in \mathbb P_{k_0}$}.
\end{align}
Indeed, for all $u\in\mathbb P_{k_0}$, by~\eqref{eq:below2} and~\eqref{eq:lambdak} we have
\begin{align*}
    \mathcal{B}_{\alpha,V}(u,u)&=\mathcal{B}_{\alpha,V}(u,u)+\mu \|u\|_2^2-\mu\|u\|_2^2\\
    &=\left(1-\frac{\mu }{\lambda_{k_0}+\mu }\right)\left(\mathcal{B}_{\alpha,V}(u,u)+\mu \|u\|_2^2\right)\\
    &\quad+\left(\frac{\mu}{\lambda_{k_0}+\mu }\right)\left(\mathcal{B}_{\alpha,V}(u,u)+\mu \|u\|_2^2\right)-\mu\|u\|_2^2\\
    &\ge c_0\,\,\|u\|_W^2
\end{align*}
where $c_0:=\frac 12\left(1-\frac{\mu }{\lambda_{k_0}+\mu }\right)$.
\begin{lemma}\label{condgeome2}
Assume the validity of the assumptions of Theorem~\ref{mainthm2}.
Then
\begin{itemize}
    \item when $p\in(4,6)$, the functional $\mathcal{J}$ satisfies $(i)$ and $(ii)$ of Theorem~\ref{Rabinowitz86} in $X=W$, with $X_1=H_{k_0}$ and $X_2=\mathbb P_{k_0}$, where $\lambda_{k_0}$ is given by~\eqref{k0};
    \item when $p=4$, the functional $\mathcal{J}$ satisfies $(i)$ and $(ii)$ of Theorem~\ref{thm:AR} in $X=W$.
\end{itemize}
\end{lemma}

\begin{proof}
We divide the proof in two steps.

\noindent{\bf Step 1.} The first geometric condition in Theorems~\ref{Rabinowitz86} and~\ref{thm:AR}. We consider two cases.

\noindent {\bf Case 1}. Let $p\in(4,6)$, and let us prove the validity of $(i)$ in Theorem~\ref{Rabinowitz86}. By~\eqref{eq:J2} and~\eqref{eq:ck0}, for all $u\in X_2=\mathbb P_{k_0}$ we have
\begin{equation*}
    \mathcal{J}(u)\ge \frac{1}{2}\mathcal{B}_{\alpha,V}(u,u)-\frac{1}{p}\|u\|_p^p\ge\left(\frac{c_0}{2}-\frac{C_p^p}{p}\|u\|_W^{p-2}\right)\|u\|_W^2,
\end{equation*}
where $C_p>0$ is the constant satisfying
\begin{equation}\label{eq:WSobolev}
    \|u\|_p\le C_p\|u\|_W\quad\text{for all $u\in W$}.
\end{equation}
Hence, as in Lemma~\ref{condgeome1}, there exist $\delta,\varrho>0$ such that
\begin{equation*}
    \inf \mathcal{J}(S_\varrho\cap X_2)\ge\delta,
\end{equation*}
where $S_\varrho:=\{u\in W\,:\, \|u\|_W=\varrho\}$.

\noindent {\bf Case 2}. Assume that $p=4$. We claim that there exist $\delta,\varrho>0$ such that
\begin{equation*}
    \inf \mathcal{J}(S_\varrho) \ge \delta.
\end{equation*}
Indeed, by~\eqref{young}, for any $u\in W$ and $\varepsilon>0$ we have
\begin{align*}
&\dfrac{\hbar^2}{4m}\mathcal{B}_{\alpha}(u,u)+\frac{1}{2}\int_{\R^3}Vu^2\,\d x\\
&\ge \dfrac{\hbar^2}{4m}\|\nabla u\|_2^2-\dfrac{\alpha^-\hbar^2}{4m}\left(s\varepsilon \|\nabla u\|_2^2+(1-s)\varepsilon^{-\frac{s}{1-s}}\|u\|^2_2\right)+\frac{1}{2}\int_{\R^3}\!\!Vu^2\d x\\
    &\ge \dfrac{\hbar^2}{4m}\left(1-\alpha^{-}s\varepsilon\right)\|\nabla u\|_2^2+\frac{1}{2}\left(1-\dfrac{\alpha^-\hbar^2}{2mV_0}(1-s)\varepsilon^{-\frac{s}{1-s}}\right)\int_{\R^3}\!\!Vu^2\d x.
\end{align*}
Let us consider the following system
\begin{equation}\label{eq:system2}
\begin{cases}
    1-\alpha^-s\varepsilon>0,\\
    1-\dfrac{\alpha^-\hbar^2}{2mV_0}(1-s)\varepsilon^{-\frac{s}{1-s}}>0.
\end{cases}
\end{equation}
As in the proof of Lemma~\ref{condgeome1}, we derive that the system~\eqref{eq:system2} is satisfied when either $\alpha^-=0$ or $\alpha^->0$ and
$$\frac{(1-s)^{\frac{1-s}{s}}\left(\hbar^2\right)^{\frac{1-s}{s}}\left(\alpha^-\right)^{\frac{1-s}{s}}}{(2mV_0)^{\frac{1-s}{s}}}<\varepsilon<\frac{1}{\alpha^-s}.$$
Since $\alpha$ satisfies~\eqref{eq:alpha2}, which implies
\begin{equation*}
    \alpha^{-}<\left(\frac{2mV_0}{\hbar^2}\right)^{1-s}(1-s)^{s-1}s^{-s},
\end{equation*}
there exists $\varepsilon_0\in(0,\infty)$ such that
\begin{equation*}
    c_1:=\frac{\hbar^2}{2m}(1-\alpha^{-}s\varepsilon_0)>0\quad\text{and}\quad c_2:=1-\dfrac{\alpha^-\hbar^2}{2mV_0}(1-s)\varepsilon_0^{-\frac{s}{1-s}}>0.
\end{equation*}
Hence we get
\begin{equation}\label{eq:below3}
\begin{aligned}
    \dfrac{\hbar^2}{2m}\mathcal{B}_{\alpha}(u,u)+\int_{\R^3}Vu^2\,\d x&\ge\frac{1}{2}\min\{c_1,c_2\}\left(\|\nabla u\|^2_2+\int_{\R^3}Vu^2\,\d x\right)\\
    &\ge\frac{1}{2}\min \{c_1,c_2\}\min\{1,V_0\}\|u\|^2_W.
\end{aligned}
\end{equation}
Therefore, by using also Lemma~\ref{lem:Phiu}--(i) and ~\eqref{eq:WSobolev}, for any $u\in S_{\varrho}$, we have
\begin{align*}
    \mathcal{J}(u)&=\dfrac{\hbar^2}{4m}\mathcal{B}_{\alpha}(u,u)+\frac{1}{2}\int_{\R^3}Vu^2\,\d x+\dfrac{1}{4}\int_{\R^3}\Phi_u u^2\,\d x-\frac{1}{4}\|u\|^4_4\\
    &\ge\frac{1}{2}\min \{c_1,c_2\}\min\{1,V_0\}\|u\|^2_W-\frac{C_4^4}{4}\|u\|^4_W\\
&= \varrho^2 \left( \frac{1}{2}\min \{c_1,c_2\}\min\{1,V_0\}-\frac{C_4^4}{4}\cdot \varrho^2\right)>0,
\end{align*}
where the last inequality is given by eventually setting
\begin{equation*}
  \varrho<\left(\frac{2\min \{c_1,c_2\}\min\{1,V_0\}}{C_4p^4}\right)^\frac{1}{2}.
\end{equation*}
Therefore, condition $(i)$ is satisfied also for Theorem~\ref{thm:AR}.

\medskip

\noindent{\bf Step 2.} The second geometric condition in Theorems~\ref{Rabinowitz86} and~\ref{thm:AR}. To prove $(ii)$ we consider again two cases.

\noindent{\bf Case 1.} Let $p\in (4,6)$ and let us prove $(ii)$ of Theorem~\ref{Rabinowitz86}. By~\eqref{young} and~\eqref{eq:key}, there exists a constant $K>0$ such that, for all finite dimensional space $Y\subset W$ and $u\in Y$, we have
\begin{equation*}
    \mathcal{J}(u)\le K\|u\|_W^2+\frac{\pi}{2}C_D^2\|u\|_{\frac{12}{5}}^4-\frac{1}{p}\|u\|_p^p\to-\infty
\end{equation*}
as $\|u\|_W\to \infty$, since all the norms are equivalent.

\noindent{\bf Case 2.}  Let $p=4$. To prove $(ii)$ of Theorem~\ref{thm:AR}, we fix $u\in{\rm\bf C}_c^\infty(\R^3)\setminus\{0\}$, $\lambda>1$, and $\beta,\gamma\in\R$, and consider $u_{\lambda,\beta,\gamma}$, introduced in Definition~\ref{def:u-lambda}. We notice that $u_{\lambda,\beta,\gamma}\in{\rm\bf C}_c^\infty(\R^3)\subset W$ and fix $L>0$ such that ${\rm supp\,} u\subseteq B_L(0)$.
Hence, by Lemma~\ref{lem:stime-lambda}, we have
\begin{align*}
    \mathcal{J}(u_{\lambda,\beta,\gamma})&\le\frac{\hbar^2}{4m}\lambda^{2\gamma-\beta}\|\nabla u\|_2^2+\frac{\alpha^+\hbar^2}{4m}\lambda^{2\gamma+2s\beta-3\beta}\|(-\Delta)^{\frac{s}{2}}u\|_2^2\\
    &\quad+\frac{1}{2}\lambda^{2\gamma-3\beta}\int_{B_L(0)}V(\lambda^{-1}x)|u(x)|^2\,\d x+\frac{1}{4}\lambda^{4\gamma-5\beta}\int_{\R^3}\Phi_u u^2\,\d x\\
    &\quad-\frac{1}{4}\lambda^{4\gamma-3\beta}\|u\|_4^4\\
    &\le\frac{\hbar^2}{4m}\lambda^{2\gamma-\beta}\|\nabla u\|_2^2+\frac{\alpha^+\hbar^2}{4m}\lambda^{2\gamma+2s\beta-3\beta}\|(-\Delta)^{\frac{s}{2}}u\|_2^2\\
    &\quad+\frac{V_L}{2}\lambda^{2\gamma-3\beta}\|u\|_2^2+\frac{1}{4}\lambda^{4\gamma-5\beta}\int_{\R^3}\Phi_u u^2\,\d x-\frac{1}{4}\lambda^{4\gamma-3\beta}\|u\|_4^4,
\end{align*}
where $V_L:=\max_{x\in B_L(0)}V(x)$. We want to prove that $\mathcal{J}(u_{\lambda,\beta,\gamma})\le 0$ for some suitable choice of $\lambda>1$ and $\beta,\gamma\in\R$.

For example, if we assume that ~\eqref{system} holds,
then we derive  that $\mathcal{J}(u_{\lambda,\beta,\gamma})\to -\infty$ as $\lambda\to \infty$.
The proof can then be completed exactly as the proof of Lemma~\ref{condgeome1}.
\end{proof}
\begin{lemma}\label{PS2}
Under the assumptions of Theorem~\ref{mainthm2}, the functional $\mathcal{J}$ satisfies $(iii)$ in Theorems~\ref{Rabinowitz86} and~\ref{thm:AR}.
\end{lemma}
\begin{proof}
Let us fix a $(PS)$ sequence $(u_n)_n\subset W$. Then, for all $p\in[4,6)$ we have
\begin{equation*}
\begin{aligned}
  &p\mathcal{J}(u_n)-\mathcal{J}'(u_n)[u_n]\\
  &\ge\left(\frac{p}{2}-1\right)\left(\dfrac{\hbar^2}{2m}\mathcal{B}_\alpha(u_n,u_n)+\int_{\R^3}Vu_n^2\,\d x\right)+\left(\frac{p}{4}-1\right)\int_{\R^3}\Phi_{u_n}u_n^2\,\d x\\
  &\ge\left(\frac{p}{2}-1\right)\mathcal{B}_{\alpha,V}(u_n,u_n).
\end{aligned}
\end{equation*}

We distinguish between two possible cases.
\medskip

\noindent {\bf Case 1}
Let $p\in(4,6)$.
Then, by~\eqref{eq:below2},
\begin{equation*}
  p\mathcal{J}(u_n)-\mathcal{J}'(u_n)[u_n]\ge \left(\frac{p}{2}-1\right)\mathcal{B}_{\alpha,V}(u_n,u_n)\ge d_1\|u_n\|_W^2-d_2\|u_n\|_2^2,
\end{equation*}
where
\[
d_1:=\frac{1}{2}\left(\frac{p}{2}-1\right)\quad\text{and}\quad d_2:=\mu\left(\frac{p}{2}-1\right).
\]
\noindent Assume, by contradiction, that $\|u_n\|_W\to \infty$ and define $w_n:=\frac{u_n}{\|u_n\|_W}$ for all $n\in\N$. Since $\|w_n\|_W=1$ for all $n\in\N$, by Lemma~\ref{lem:com} there exist a subsequence, not relabeled, and a function $w\in W$ such that, as $n\to\infty$, $(w_n)_n$ converges to $w$ weakly in $W$ and strongly in $L^p(\R^3)$ for all $p\in [2,6)$.
In particular, since
\begin{equation*}
d_1-d_2\|w_n\|_2^2\le \frac{p\mathcal{J}(u_n)}{\|u_n\|_W^2}-\frac{\mathcal{J}'(u_n)[u_n]}{\|u_n\|_W^2}\to 0
\end{equation*}
as $n\to\infty$, we have
\begin{equation*}
\|w\|_2^2\ge\frac{d_1}{d_2}>0.
\end{equation*}
On the other hand, by~\eqref{eq:key} and~\eqref{eq:J2}, for all $n\in\N$ we have
\begin{equation*}
\frac{1}{p}\|u_n\|_p^p\le \frac{1}{2}\mathcal{B}_{\alpha,V}(u_n,u_n)+\frac{\pi}{2}C_D^2\|u_n\|_{\frac{12}{5}}^4-\mathcal{J}(u_n).
\end{equation*}
Therefore, since $(u_n)_n$ is a $(PS)$ sequence, there exist constants $d_3,d_4,d_5>0$ such that
\begin{align*}
    0<\frac{1}{p}\|w_n\|_p^p&\le \frac{1}{2}\frac{\mathcal{B}_{\alpha,V}(u_n,u_n)}{\|u_n\|_W^p}+\frac{\pi}{2}C_D^2\frac{\|u_n\|_{\frac{12}{5}}^4}{\|u_n\|_W^p}+\frac{|\mathcal{J}(u_n)|}{\|u_n\|_W^p}\\
    &\le \frac{d_3}{\|u_n\|_W^{p-2}}+\frac{d_4}{\|u_n\|_W^{p-4}}+\frac{d_5}{\|u_n\|_W^p}\to 0
\end{align*}
as $n\to\infty$, being $p>4$.
Hence $w=0$, which leads to a contradiction. Thus $(u_n)_n\subset W$ is bounded for $p\in(4,6)$.
\medskip

\noindent {\bf Case 2}.
Let $p=4$.
Then, $V_0>0$ and $\alpha$ satisfies~\eqref{eq:alpha2}.
Therefore, by~\eqref{eq:below3}, we have
\begin{align*}
    p\mathcal{J}(u_n)-\mathcal{J}'(u_n)[u_n]&\ge\left(\frac{p}{2}-1\right)\mathcal{B}_{\alpha,V}(u_n,u_n)\\
    &\ge\left(\frac{p}{2}-1\right)\min\{c_1,c_2\}\min\{1,V_0\}\|u_n\|_W^2.
\end{align*}
Since $(\mathcal{J}(u_n))_n$ is bounded in $\R$ and $(\mathcal{J}'(u_n))_n$ is bounded in $W'$, being $(u_n)_n$ a $(PS)$ sequence, there exist two positive constants $K_1,K_2$ such that
\begin{equation}\label{PSseq2}
\mathcal{J}(u_n)\le\ K_1\quad\text{and}\quad|\mathcal{J}'(u_n)[u_n]|\le K_2\|u_n\|_W\quad\text{for any }n\in\mathbb{N}.
\end{equation}
Hence, by~\eqref{eq:below3} and~\eqref{PSseq2}, we get
\begin{equation*}
  pK_1+K_2\|u_n\|_W\ge \left(\frac{p}{2}-1\right)\min\{c_1,c_2\}\min\{1,V_0\}\|u_n\|_W^2\quad\text{for any }n\in\mathbb{N},
\end{equation*}
which implies that $(u_n)_n$ is bounded in $W$.

\noindent Therefore, there exist a subsequence, not relabeled, and $u\in W$ such that $(u_n)_n$ converges to $u$ weakly in $W$, strongly in $L^p(\R^3)$ for any $p\in[2,6)$, and a.e. in $\R^3$.

To conclude we show that the convergence in $W$ turns out to be strong. By~\eqref{eq:key2}, the sequence $(\Phi_{u_n})_n\subset\mathcal{D}^{1,2}(\R^3)$ is bounded in $\mathcal{D}^{1,2}(\R^3)$.
Moreover, by~\eqref{eq:below2}, we have
\begin{equation}\label{stimePS2}
\begin{split}
    \frac{1}{2}\|u_n-u\|_W^2&\le\mathcal{B}_{\alpha,V}(u_n,u_n)+\mu\|u_n-u\|_2^2\\
    &=\mathcal{J}'(u_n)[u_n-u]-\mathcal{J}'(u)[u_n-u]+\mu\|u_n-u\|_2^2\\
    &\quad-\int_{\R^3}(\Phi_{u_n}u_n-\Phi_u u)(u_n-u)\,\d x\\
    &\quad+\int_{\R^3}(|u_n|^{p-2}u_n-|u|^{p-2}u)(u_n-u)\,\d x.
\end{split}
\end{equation}
Since $\mathcal{J}'(u_n)\to 0$ in $W'$, $u_n\rightharpoonup u$ in $W$, and $u_n\to u$ in $L^2(\R^3)$ as $n\to\infty$, it follows that the first three terms in the right hand side of~\eqref{stimePS2} converge to 0 as $n\to\infty$.
Moreover, as $n\to\infty$
\begin{align*}
    \left|\int_{\R^3}(\Phi_{u_n}u_n-\Phi_u u)(u_n-u)\,\d x\right|\le\left(\|\Phi_{u_n}\|_6\|u_n\|_\frac{12}{5}+\|\Phi_u\|_6\|u\|_\frac{12}{5}\right)\|u_n-u\|_\frac{12}{5}\!\!\to 0,\\
    \left|\int_{\R^3}(|u_n|^{p-2}u_n-|u|^{p-2}u)(u_n-u)\,\d x\right|\le(\|u_n\|_p^{p-1}+\|u_n\|_p^{p-1})\|u_n-u\|_p\to 0.
\end{align*}
Hence the thesis follows.
\end{proof}
We can now conclude this section by collecting all the results given in the previous lines in the proof of Theorem~\ref{mainthm2}.
\begin{proof}[Proof of Theorem~\ref{mainthm2}]
The proof follows by Lemmas~\ref{condgeome2}--\ref{PS2} and by Theorems~\ref{Rabinowitz86} and~\ref{thm:AR}.
\end{proof}

%%%%%%%%%%%%%%%%%%%%

\section*{Declarations}

\bigskip
\subsection*{Ethics approval and consent to participate}
Not applicable.

\subsection*{Consent for publication}
Not applicable.

\subsection*{Availability of data and materials}
Not applicable.

\subsection*{Competing Interest}
The authors have no conflicts of interest to declare.

\subsection*{Funding}
The authors are members and acknowledge the support of {\it Gruppo Nazionale per l’Analisi Matematica, la Probabilità e le loro Applicazioni} (GNAMPA) of {\it Istituto Nazionale di Alta Matematica} (INdAM).

N.~Cangiotti acknowledges the support of the MIUR - PRIN 2017 project ``From Models to Decisions" (Prot. N. 201743F9YE).

M.~Caponi acknowledges the support of the project STAR PLUS 2020 - Linea~1 (21-UNINA-EPIG-172) ``New perspectives in the Variational modeling of Continuum Mechanics'', and of the MIUR - PRIN 2017 project ``Variational Methods for Stationary and Evolution Problems with Singularities and Interfaces'' (Prot. N.~2017BTM7SN).

A.~Maione acknowledges the support of the DFG SPP 2256 project ``Variational Methods for Predicting Complex Phenomena in Engineering Structures and Materials'' (Project number 422730790) and the support of the University of Freiburg.

M.~Caponi, A.~Maione and E.~Vitillaro are also supported by the INdAM - GNAMPA Project ``Equazioni differenziali alle derivate parziali di tipo misto o dipendenti da campi di vettori'' (Project number CUP\_E53C22001930001).

E.~Vitillaro is supported by the MIUR - PRIN 2022 project ''Advanced theoretical aspects in PDEs and their applications'' (Prot. N.  2022BCFHN2) and by``Progetti Equazioni delle onde con condizioni iperboliche ed acustiche al bordo,  finanziati  con  i Fondi  Ricerca  di Base 2017--2022, della Universit\`a degli Studi di Perugia''.

% A. Maione acknowledges the support provided to the Centre de Recerca Matemàtica by the María de Maeztu Award. This work is supported by the Spanish State Research Agency, through the Severo Ochoa and María de Maeztu Program for Centers and Units of Excellence in R\&D (CEX2020-001084-M).
% He also thanks CERCA Programme/Generalitat de Catalunya for institutional support.

\subsection*{Authors' contributions}
All  authors have equally contributed to the work.

\subsection*{Acknowledgements}
The authors wish to thank Dimitri Mugnai for many useful discussions.

%---------------------------------
% Bibliography
%---------------------------------

\end{document}